\documentclass[11pt]{article}

\usepackage[colorlinks, linkcolor=blue, citecolor=blue]{hyperref}            
\usepackage{color}
\usepackage{graphicx,subfigure,amsmath,amssymb,amsfonts,bm,epsfig,epsf,url,dsfont}
\usepackage{amsthm}
\usepackage{tikz}
\usepackage{times}
\usepackage{bbm}      
\usepackage{booktabs}
\usepackage{cases}
\usepackage{fullpage}
\usepackage[small,bf]{caption}
\usepackage{natbib}
\usepackage[top=1in,bottom=1in,left=1in,right=1in]{geometry}
\usepackage{fancybox}
\usepackage[bottom]{footmisc}

\DeclareMathOperator*{\Motimes}{\text{\raisebox{0.25ex}{\scalebox{0.8}{$\bigotimes$}}}}

\newcommand{\mD}{\mathcal{D}}
\newcommand{\mL}{\mathcal{L}}

\newcommand{\mG}{\mathcal{G}}
\newcommand{\mN}{\mathcal{N}}

\newcommand{\mP}{\mathcal{P}}
\newcommand{\mQ}{\mathcal{Q}}

\newcommand{\TV}{d_{\text{TV}}}
\newcommand{\KL}{d_{\text{KL}}}
\newcommand{\SKL}{d_{\text{SKL}}}

\newcommand{\bP}{\mathbb{P}}
\newcommand{\bE}{\mathbb{E}}

\newcommand{\pr}[1]{\textsc{#1}}

\newtheorem{question}{Question}[section]
\newtheorem{definition}{Definition}[section]

\newtheorem{fact}{Fact}[section]
\newtheorem{proposition}{Proposition}[section]
\newtheorem{theorem}{Theorem}[section]
\newtheorem{corollary}{Corollary}[section]
\newtheorem{conjecture}{Conjecture}[section]
\newtheorem{lemma}[theorem]{Lemma}

\newenvironment{fminipage}%
  {\begin{Sbox}\begin{minipage}}%
  {\end{minipage}\end{Sbox}\fbox{\TheSbox}}

\newenvironment{algbox}[0]{\vskip 0.2in
\noindent 
\begin{fminipage}{6.3in}
}{
\end{fminipage}
\vskip 0.2in
}

\begin{document}

\title{Universality of Computational Lower Bounds \\ for Submatrix Detection}

\author{Matthew Brennan\thanks{Massachusetts Institute of Technology. Department of EECS. Email: \texttt{brennanm@mit.edu}.}
\and 
Guy Bresler\thanks{Massachusetts Institute of Technology. Department of EECS. Email: \texttt{guy@mit.edu}.}
\and
Wasim Huleihel\thanks{Tel-Aviv University. Department of EE. Email: \texttt{wasim8@gmail.com}.}}
\date{\today}

\maketitle

\begin{abstract}
In the general submatrix detection problem, the task is to detect the presence of a small $k \times k$ submatrix with entries sampled from a distribution $\mP$ in an $n \times n$ matrix of samples from $\mQ$. This formulation includes a number of well-studied problems, such as biclustering when $\mP$ and $\mQ$ are Gaussians and the planted dense subgraph formulation of community detection when the submatrix is a principal minor and $\mP$ and $\mQ$ are Bernoulli random variables. These problems all seem to exhibit a universal phenomenon: there is a statistical-computational gap depending on $\mP$ and $\mQ$ between the minimum $k$ at which this task can be solved and the minimum $k$ at which it can be solved in polynomial time.

Our main result is to tightly characterize this computational barrier as a tradeoff between $k$ and the KL divergences between $\mP$ and $\mQ$ through average-case reductions from the planted clique conjecture. These computational lower bounds hold given mild assumptions on $\mP$ and $\mQ$ arising naturally from classical binary hypothesis testing. In particular, our results recover and generalize the planted clique lower bounds for Gaussian biclustering in \cite{ma2015computational, brennan2018reducibility} and for the sparse and general regimes of planted dense subgraph in \cite{hajek2015computational, brennan2018reducibility}. This yields the first universality principle for computational lower bounds obtained through average-case reductions.

To reduce from planted clique to submatrix detection for a specific pair $\mP$ and $\mQ$, we introduce two techniques for average-case reductions: (1) multivariate rejection kernels which perform an algorithmic change of measure and lift to a larger submatrix while obtaining an optimal tradeoff in KL divergence, and (2) a technique for embedding adjacency matrices of graphs as principal minors in larger matrices that handles distributional issues arising from their diagonal entries and the matching row and column supports of the $k \times k$ submatrix. We suspect that these techniques have applications in average-case reductions to other problems and are likely of independent interest. We also characterize the statistical barrier in our general formulation of submatrix detection.
\end{abstract}

\pagebreak

\tableofcontents

\pagebreak

\section{Introduction}

In the general submatrix detection problem, the task is to detect the presence of a small $k \times k$ submatrix with entries sampled from a distribution $\mP$ in an $n \times n$ matrix of samples from $\mQ$. This problem arises in many natural contexts for specific pairs of distributions $(\mP, \mQ)$. When $\mP$ and $\mQ$ are Gaussians, this yields the well-studied problem of biclustering arising from applications in analyzing microarray data \cite{shabalin2009finding}. A large body of work has studied the information-theoretic lower bounds, algorithms and limitations of restricted classes of algorithms for biclustering \cite{butucea2013detection, montanari2015limitation, shabalin2009finding, kolar2011minimax, balakrishnan2011statistical, chen2016statistical, cai2017computational}. When the $k \times k$ submatrix is a principal minor and $\mP$ and $\mQ$ are Bernoulli random variables, general submatrix detection becomes the planted dense subgraph formulation of community detection. This problem has also been studied extensively from algorithmic and information-theoretic viewpoints \cite{arias2014community, butucea2013detection, verzelen2015community, chen2016statistical, montanari2015finding, candogan2018finding, hajek2016achieving}.

The best known algorithms for both the Gaussian and Bernoulli problems seem to exhibit a peculiar phenomenon: there appears to be a statistical-computational gap between the minimum $k$ at which this task can be solved and the minimum $k$ at which it can be solved in polynomial time. Tight statistical-computational gaps for both biclustering and several parameter regimes of planted dense subgraph were recently established through average-case reductions from the planted clique conjecture \cite{ma2015computational, hajek2015computational, brennan2018reducibility}. Furthermore, the regimes in which these problems are information-theoretically impossible, statistically possible but computational hard and admit polynomial time algorithms appear to have a common structure. This raises the following natural question:

\begin{question}
Are the statistical-computational gaps for general submatrix detection a universal phenomenon regardless of the specific pair of distributions $(\mP, \mQ)$?
\end{question}

We answer this question for a wide class of pairs of distributions $(\mP, \mQ)$. Our main result is to tightly characterize this computational barrier as a tradeoff between $k$ and the KL divergences between $\mP$ and $\mQ$ through average-case reductions from the planted clique conjecture. These computational lower bounds hold given mild assumptions on $\mP$ and $\mQ$ arising naturally from classical binary hypothesis testing. In particular, our results recover and widely generalize the planted clique lower bounds for Gaussian biclustering in \cite{ma2015computational, brennan2018reducibility} and for the sparse and general regimes of planted dense subgraph in \cite{hajek2015computational, brennan2018reducibility}. This yields the first universality principle for computational lower bounds obtained through average-case reductions. We also characterize the statistical barrier in our general formulation of submatrix detection.

Average-case reductions are notoriously brittle in the sense that most natural maps designed for worst-case problems fail to faithfully map a natural distribution to a natural distribution over the target problem. A universality result obtained through an average-case reduction necessarily overcomes this barrier in a strong way as it would have to simultaneously map to an entire set of natural distributions over the target problem. A main contribution of the paper is to introduce techniques handling subtle technical obstacles that arise when devising such a reduction.

Our results are close in flavour to several previous works showing universal phenomena in the context of submatrix problems. In \cite{montanari2015finding}, approximate message passing algorithms were shown to recover the support of the planted submatrix under regularity conditions on $(\mP, \mQ)$. \cite{hajek2016semidefinite} analyzed semidefinite programming algorithms also under regularity conditions on $(\mP, \mQ)$. In \cite{hajek2017information}, the information-theoretic thresholds for submatrix localization -- the recovery variant of our detection problem -- were shown under very mild assumptions on $(\mP, \mQ)$, characterizing the statistical limit of the problem over a large universality class. In \cite{lesieur2015mmse, deshpande2015asymptotic, krzakala2016mutual}, information-theoretic thresholds for universal formulations of rank-one matrix estimation were established by analyzing approximate message passing and Guerra's interpolation argument. Our work is the first to analyze both the information-theoretic and computational lower bounds for submatrix detection universally and the first work we are aware of producing a universality class of computational lower bounds for any problem. The conditions on $(\mP, \mQ)$ for our computational lower bounds are mild and similar to those in \cite{hajek2017information}.

\subsection{Contributions to Techniques for Average-Case Reductions}

One of our main contributions is to introduce two new techniques for average-case reductions that are of independent interest. This work is part of a growing body of literature establishing statistical-computational gaps in high-dimensional inference problems based on average-case reductions. Previous reductions include lower bounds for testing $k$-wise independence \cite{alon2007testing}, RIP certification \cite{wang2016average, koiran2014hidden}, matrix completion \cite{chen2015incoherence} and sparse PCA \cite{berthet2013optimal, berthet2013complexity, wang2016statistical, gao2017sparse, brennan2019optimal}. A number of techniques were introduced in \cite{brennan2018reducibility} to provide the first web of average-case reductions to problems including planted independent set, planted dense subgraph, sparse spiked Wigner, sparse PCA, the subgraph stochastic block model and biclustering. More detailed surveys of this area can be found in the introduction section of \cite{brennan2018reducibility} and in \cite{wu2018statistical}. In this work, we introduce the following two techniques to map from a generalization of planted clique to general submatrix detection: 
\begin{itemize}
\item \textbf{Planting Diagonals by Embedding as a Minor}: This is a technique for embedding adjacency matrices of graphs as principal minors in larger matrices. It handles distributional issues arising from missing diagonal entries in adjacency matrices and the matching row and column supports of the $k \times k$ submatrix.
\item \textbf{Multivariate Rejection Kernels:} These are randomized maps that perform an algorithmic change of measure and lift to a larger submatrix while obtaining an optimal tradeoff in KL divergence that matches the target lower bounds for submatrix detection.
\end{itemize}
The first technique solves a central obstacle in the reductions of \cite{ma2015computational}, \cite{hajek2015computational}, and \cite{brennan2018reducibility} to biclustering and planted dense subgraph. In \cite{hajek2015computational}, it is noted that the main issue leading to the complicated analysis of their reduction arises from the missing diagonal entries in the adjacency matrix of planted clique, which on lifting get mapped to ``holes'' in the community. In \cite{brennan2018reducibility}, this same obstacle was overcome through $\textsc{Distributional-Lifting}$, an involved technique first performing an algorithmic change of measure and then iteratively lifting the resulting problem. This method relied crucially on the existence of a ``cloning'' map for the new pair of measures $(\mP, \mQ)$, which were exhibited in the Poisson and Gaussian cases. In \cite{ma2015computational}, this obstacle was handled by restricting to a lower left submatrix of the adjacency matrix, but in doing so broke the symmetry in the row and column supports of the $k \times k$ submatrix.

None of these previous reductions generalize to the universal setting with arbitrary $(\mP, \mQ)$. Our first technique resolves the missing diagonal entries obstacle cleanly and is crucial to yielding lower bounds for arbitrary $(\mP, \mQ)$ and producing a general submatrix instance with matching row and column supports. The latter property is essential to our reduction implying lower bounds for planted dense subgraph. These obstacles and our first technique are described in more detail in Section \ref{sec:avgredsub}.

Our second technique generalizes the rejection kernel framework introduced in \cite{brennan2018reducibility}, while simultaneously performing a lift to a higher dimensional instance. This technique overcomes issues that prior reductions and a seemingly natural approach face when mapping to general $(\mP, \mQ)$. The reductions in \cite{ma2015computational} and \cite{hajek2015computational} for performing algorithmic changes of measure required sampling explicit PMFs. This is infeasible when lifting to a higher dimensional instance, as we need to do in order to obtain tight lower bounds for general submatrix detection. $\pr{Distributional-Lifting}$ in \cite{brennan2018reducibility} required an efficient ``cloning'' map, which does not clearly exist for general $(\mP, \mQ)$. One natural approach to showing hardness for submatrix detection is to first show hardness for one pair $(\mP', \mQ')$ and then reduce entrywise to the target $(\mP, \mQ)$. We show in Section \ref{sec:entrywise} that this approach fails to show tight lower bounds. Multivariate rejection kernels bypass all of these issues, as we discuss further in Section \ref{sec:mrklarge}. 

\subsection{Outline of the Paper}

The paper is structured as follows. In Section \ref{sec:defns}, we formally define the general submatrix problem and motivate some assumptions on $(\mP, \mQ)$ from classical binary hypothesis testing. In Section \ref{sec:summary}, we give general statements of our universality results for computational and statistical barriers. We also specialize these results to universality classes over which we can obtain complete characterizations of the computational phase diagram for submatrix detection. In Section \ref{sec:avgred}, we provide preliminaries on average-case reductions in total variation. In Section \ref{sec:mrklarge}, we introduce and analyze multivariate rejection kernels. In Section \ref{sec:avgredsub}, we give our general average-case reduction $\textsc{To-Submatrix}$. In Section \ref{sec:compbarriers}, we deduce computational lower bounds from this reduction and analyze simple test statistics showing achievability of these lower bounds. In Section \ref{sec:statbarriers}, we establish the statistical limits of submatrix detection. In Section \ref{sec:universalityclasses}, we discuss the strength of the assumptions giving rise to our main three universality classes $\pr{uc-a}, \pr{uc-b}$ and $\pr{uc-c}$ and the distributions that they contain. In Section \ref{sec:conc}, we discuss open problems remaining after this work.

\subsection{Notation}

In this paper, we adopt the following notation. Let $\mL(X)$ denote the distribution law of a random variable $X$ and given two laws $\mL_1$ and $\mL_2$, let $\mL_1 + \mL_2$ denote $\mL(X + Y)$ where $X \sim \mL_1$ and $Y \sim \mL_2$ are independent. Given a distribution $\mathcal{P}$, let $\mathcal{P}^{\otimes n}$ denote the distribution of $(X_1, X_2, \dots, X_n)$ where the $X_i$ are i.i.d. according to $\mathcal{P}$. Similarly, let $\mathcal{P}^{\otimes m \times n}$ denote the distribution on $\mathbb{R}^{m \times n}$ with i.i.d. entries distributed as $\mathcal{P}$. Let $\TV$, $\KL$ and $\chi^2$ denote total variation distance, KL divergence and $\chi^2$ divergence, respectively. Let $[n] = \{1, 2, \dots, n\}$ and let $\mathbf{1}_S$ denote the vector $v \in \mathbb{R}^n$ with $v_i = 1$ if $i \in S$ and $v_i = 0$ if $i \not \in S$ where $S \subseteq [n]$.

\section{Submatrix Problems and Conditions for Universality}
\label{sec:defns}

\subsection{General Submatrix Detection}

The primary focus in this work are detection problems, wherein an algorithm is given a set of observations and tasked with distinguishing between two hypotheses:
\begin{itemize}
\item a \emph{uniform} hypothesis $H_0$, under which observations are generated from the natural noise distribution for the problem; and
\item a \emph{planted} hypothesis $H_1$, under which observations are generated from the same noise distribution but modified by planting a latent sparse structure.
\end{itemize}
In the problems we consider, $H_0$ and $H_1$ are typically both simple hypothesis consisting of a single distribution. As discussed in \cite{brennan2018reducibility} and \cite{hajek2015computational}, lower bounds for simple vs. simple hypothesis testing formulations are stronger and technically more difficult than for formulations involving composite hypotheses. For a given detection problem, the goal is to design an algorithm $\mathcal{A}(X) \in \{0, 1\}$ that classifies an input $X$ with low asymptotic Type I$+$II error
$$\limsup_{n \to \infty} \left\{ \bP_{H_0}[\mathcal{A}(X) = 1] + \bP_{H_1}[\mathcal{A}(X) = 0] \right\}$$
where $n$ is the parameter indicating the size of $X$. If the asymptotic Type I$+$II error of $\mathcal{A}$ is zero, then we say $\mathcal{A}$ \emph{solves} the detection problem. We now define the universal formulation of submatrix detection that will be our main object of study. Throughout this paper, $(\mP, \mQ)$ will either denote a fixed pair of distributions over a measurable space $(X, \mathcal{B})$ or, when there is a natural problem parameter $n$, implicitly denote a pair of sequences of distributions $\mP = (\mP_n)$ and $\mQ = (\mQ_n)$.

\begin{definition}[General Symmetric Index Set Submatrix Detection]
Given a pair of distributions $(\mP, \mQ)$ over a measurable space $(X, \mathcal{B})$, let $\pr{ssd}(n, k, \mP, \mQ)$ denote the hypothesis testing problem with observation $M \in X^{n \times n}$ and hypotheses
$$H_0 : M \sim \mQ^{\otimes n \times n} \quad \textnormal{and} \quad H_1 : M \sim \mathcal{M}(n, k, \mP, \mQ)$$
where $\mathcal{M}(n, k, \mP, \mQ)$ is the distribution of matrices $M$ with entries $M_{ij} \sim \mP$ if $i, j \in S$ and $M_{ij} \sim \mQ$ otherwise that are conditionally independent given $S$, which is chosen uniformly at random over all $k$-subsets of $[n]$.
\end{definition}

Similarly, asymmetric index set submatrix detection $\textsc{asd}(n, k, \mP, \mQ)$ is formulated with $M_{ij} \sim \mP$ for all $(i, j) \in S \times T$ where both $S$ and $T$ are chosen independently and uniformly at random over all $k$-subsets of $[n]$. Note that both information-theoretic and computational lower bounds for $\textsc{ssd}$ are stronger than for $\textsc{asd}$, since an instance of $\textsc{asd}$ can be obtained from $\textsc{ssd}$ by randomly permuting its column indices. Similarly, algorithms for $\textsc{asd}$ are stronger since permuting the columns of $\textsc{ssd}$ and applying an $\textsc{asd}$ blackbox yields an algorithm for $\textsc{ssd}$. In this work, we characterize the statistical and computational barriers in these and related problems through average-case reductions from the planted clique conjecture. For this to be possible, it is necessary to impose assumptions on $(\mP, \mQ)$ so that submatrix detection is well-posed. The next section is devoted to identifying several reasonable and natural assumptions on $(\mP, \mQ)$.

\subsection{Natural Assumptions from Classical Binary Hypothesis Testing}

Our objective is to examine the statistical-computational tradeoffs that arise in submatrix detection as a high-dimensional problem with hidden structure. More precisely, we aim to capture the tradeoff between the dimension $k$ of the hidden submatrix and how distinguishable the two distributions $\mP$ and $\mQ$ are. For this to be possible, the problem of testing between $\mP$ and $\mQ$ needs to be well-posed. Consider the classical binary hypothesis testing formulation of this task with i.i.d. observations $X_1, X_2, \dots, X_m$ where
$$H_0 : X_1, X_2, \dots, X_m \sim_{\textnormal{i.i.d.}} \mQ \quad \textnormal{and} \quad H_1 : X_1, X_2, \dots, X_m \sim_{\textnormal{i.i.d.}} \mP$$
Note that given the latent submatrix indices $S$, the problem of distinguishing between $H_0$ and $H_1$ in $\textsc{ssd}$ reduces exactly to this classical binary hypothesis testing task with $m = k^2$ samples.

In order to capture the tradeoffs that arise because of hidden structure in high dimensions, $\textsc{ssd}$ ought to be easy to solve given $S$. By the Neyman-Pearson Lemma, the optimal test is a log-likelihood ratio (LLR) test that outputs $H_1$ if
$$\sum_{i = 1}^m L(X_i) \ge m \tau \quad \textnormal{where} \quad L(x) = \log \frac{d\mP}{d\mQ}(x)$$
for some threshold $m\tau$ and where $L : X \to \mathbb{R}$ is the LLR or logarithm of the Radon-Nikodym derivative between $\mP$ and $\mQ$. We assume that $L$ can be computed efficiently so this test is computationally feasible. Note that if $\mQ$ and $\mP$ are not close to mutually absolutely continuous distributions in total variation, then there would be a non-negligible probability of seeing samples from one not in the support of the other. We assume they are exactly mutually absolutely continuous for simplicity and so $L$ is well-defined. We also assume that the expectations of the LLR with respect to $\mP$ and $\mQ$ are finite, or in other words that $\KL(\mP \| \mQ)$ and $\KL(\mQ \| \mP)$ are finite.

The error of this test resolves to the tails of the distribution of the LLR under each of $\mP$ and $\mQ$ at the threshold $m\tau$. Standard Chernoff bounds on these tails yield that if $\tau \in [-\KL(\mP \| \mQ), \KL(\mQ \| \mP)]$ then
\begin{align*}
\bP_{H_0}\left[ \sum_{i = 1}^m L(X_i) \ge m \tau \right] &\le \exp\left( - m \cdot E_{\mQ}(\tau) \right) \\
\bP_{H_1}\left[ \sum_{i = 1}^m L(X_i) < m \tau \right] &\le \exp\left( - m \cdot E_{\mP}(\tau) \right)
\end{align*}
Here, the Chernoff exponents $E_\mP, E_\mQ : \mathbb{R} \to [-\infty, \infty)$ are the Legendre transforms of the log-moment generating functions
$$E_\mQ(\tau) = \sup_{\lambda \in \mathbb{R}} \lambda \tau - \psi_\mQ(\lambda) \quad \text{and} \quad E_\mP(\tau) = \sup_{\lambda \in \mathbb{R}} \lambda \tau - \psi_\mP(\lambda)$$
where $\psi_{\mQ}(\tau) = \log \bE_\mQ[\exp(\lambda L)]$ and $\psi_{\mP}(\tau) = \log \bE_\mP[\exp(\lambda L)]$. Observe that $\psi_\mP(\lambda) = \psi_\mQ(\lambda + 1)$ and thus $E_\mQ(\tau) + \tau = E_\mP(\tau)$. Note that $\psi_\mQ'(0) = -\KL(\mP \| \mQ)$. It is well known that $\lambda \tau - \psi_\mQ(\lambda)$ is concave and has derivative at zero given by $\tau + \KL(\mP \| \mQ)$. This implies that the maximizer $\lambda^*$ to the concave optimization $E_\mQ(\tau)$ can be taken to be nonnegative if $\tau \ge -\KL(\mP \| \mQ)$. The same is true for $E_\mP(\tau)$ if $\tau \le \KL(\mQ \| \mP)$, justifying the Chernoff bounds above. We also remark that $E_\mQ$ and $E_\mP$ are nonnegative convex functions and are minimized at $E_\mQ(-\KL(\mQ \| \mP)) = E_\mP(\KL(\mP \| \mQ)) = 0$.

Note that statements of the form $E_\mQ(\tau) \ge \beta$ or $E_\mP(\tau) \ge \beta$ correspond to large deviation principles (LDP) for the LLR under $\mQ$ and $\mP$. Our main contribution is show that the tradeoff between $k$ and the KL divergence between $\mP$ and $\mQ$ dictates the computational barrier for submatrix detection as long as the LLR has LDPs under $\mQ$ and $\mP$. We devise an average-case reduction to show this given the planted clique conjecture. Our results are described in more detail in the next section. The natural assumptions arising from the discussion above are summarized in the following definition of computable pairs $(\mP, \mQ)$, that we will adopt throughout the rest of the paper.

\begin{definition}[Computable Pair of Distributions] \label{def:computable}
Define a pair of sequences of distributions $(\mP, \mQ)$ over a measurable space $(X, \mathcal{B})$ where $\mP = (\mP_n)$ and $\mQ = (\mQ_n)$ to be computable if:
\begin{enumerate}
\item there is an oracle producing a sample from $\mQ_n$ in $\textnormal{poly}(n)$ time;
\item $\mP_n$ and $\mQ_n$ are mutually absolutely continuous and the likelihood ratio satisfies
$$\bE_{x \sim \mQ_n} \left[\frac{d\mP_n}{d\mQ_n}(x) \right] = \bE_{x \sim \mP_n}\left[\left( \frac{d\mP_n}{d\mQ_n}(x) \right)^{-1} \right] = 1$$
where $\frac{d\mP_n}{d\mQ_n}$ is the Radon-Nikodym derivative.
\item the KL divergences $\KL(\mP_n \| \mQ_n)$ and $\KL(\mQ_n \| \mP_n)$ are both finite; and
\item there is an oracle computing $\frac{d\mP_n}{d\mQ_n} (x)$ in $\textnormal{poly}(n)$ time for each $x \in X$.
\end{enumerate}
\end{definition}

We assume that algorithms solving our submatrix detection problems have access to these oracles and to the values $\KL(\mP \| \mQ)$ and $\KL(\mQ \| \mP)$. The oracles appearing in this definition can be viewed as part of the computational model that we adopt. In particular, when these oracles can be implemented in computational models such as $\textbf{BPP}$, so can our reductions. We remark that the assumptions and discussion in this section are similar to the setup in \cite{hajek2017information}, which showed universality of information-theoretic lower bounds for submatrix recovery. See Sections 2.1 and 3 in \cite{hajek2017information} for further discussion of related assumptions on $\mP$ and $\mQ$.

\section{Summary of Results}
\label{sec:summary}

Our main result is an average-case reduction from planted clique showing a computational lower bound for submatrix detection in terms of KL divergence when the LLR has LDPs under $\mQ$ and $\mP$. We now briefly define the planted clique and planted dense subgraph problems as well as the planted clique and planted dense subgraph conjectures.

The planted dense subgraph problem $\pr{pds}(n, k, p, q)$ with edge densities $0 < q < p \le 1$ is the hypothesis testing problem between
$$H_0: G \sim \mG(n, q) \quad \text{and} \quad H_1 : G \sim \mG(n, k, p, q)$$
where $\mG(n, q)$ denotes an Erd\H{o}s-R\'{e}nyi random graph with edge probability $q$. Here, $\mG(n, k, p, q)$ denotes the random graph formed by sampling $\mG(n, q)$ and replacing the induced graph on a subset $S$ of size $k$ chosen uniformly at random with a sample from $\mG(k, p)$. The planted clique problem $\pr{pc}(n, k, p)$ is then $\pr{pds}(n, k, 1, p)$. There are many polynomial-time algorithms in the literature for finding the planted clique in $\mG(n, k, p)$, including approximate message passing, semidefinite programming, nuclear norm minimization and several combinatorial approaches \cite{feige2000finding, mcsherry2001spectral, feige2010finding, ames2011nuclear, dekel2014finding, deshpande2015finding, chen2016statistical}. All of these algorithms require that $k = \Omega(\sqrt{n})$ if $p$ is constant, despite the fact that the planted clique can be found by exhaustive search as soon as $k$ is larger than $2\log_{1/p} n$. This leads to the following conjecture.

\begin{conjecture}[\pr{pc} Conjecture]
Fix some constant $p \in (0, 1)$. Suppose that $\{ \mathcal{A}_n \}$ is a sequence of randomized polynomial time algorithms $\mathcal{A}_n : \mG_n \to \{0, 1\}$ and $k_n$ is a sequence of positive integers satisfying that $\limsup_{n \to \infty} \log_n k_n < \frac{1}{2}$. Then if $G$ is an instance of $\pr{pc}(n, k, p)$, it holds that
$$\liminf_{n \to \infty} \left( \bP_{H_0}\left[\mathcal{A}_n(G) = 1\right] + \bP_{H_1}\left[\mathcal{A}_n(G) = 0\right] \right) \ge 1.$$ 
\end{conjecture}

The $\pr{pc}$ Conjecture can be seen, through a simple reduction erasing random edges, to imply a similar barrier at $k = o(\sqrt{n})$ for $\pr{pds}(n, k, p, q)$ if $0 < q < p \le 1$ are constants. We refer to this as the $\pr{pds}$ Conjecture. We can now state our main computational lower bounds. Judging the quality of these bounds is easy in cases where they can be achieved by efficient algorithms. We describe a set of universality classes of $(\mP, \mQ)$ for which this is true in Section \ref{sec:completephasediagram}.

\begin{theorem}[Main Computational Lower Bounds]
Let $p \in (0, 1)$ be a fixed constant and $(\mP, \mQ)$ be a computable pair over $(X, \mathcal{B})$ such that either:
\begin{itemize}
\item $k = \Omega(\sqrt{n})$ and $\frac{k^4}{n^2} \cdot \KL(\mP \| \mQ) \to 0$ and the LLR between $(\mP, \mQ)$ satisfies the LDP
$$E_\mP\left(m \right) \ge \omega(m \log n)$$
for some positive $m$ with $\KL(\mP \| \mQ) \le m = o(n^2/k^4)$
\item $k = o(\sqrt{n})$ and $\KL(\mP \| \mQ) < \log p^{-1}$ and the LLR between $(\mP, \mQ)$ satisfies the LDP
$$E_\mP\left( \log p^{-1} \right) \ge 2\log n + \omega(1)$$
\end{itemize}
Then assuming the $\pr{pc}$ conjecture at density $p$, there is no randomized polynomial time algorithm solving $\pr{ssd}(n, k, \mP, \mQ)$ with asymptotic Type I$+$II error less than one.
\end{theorem}

This is the simplest theorem statement of our lower bounds. A more general computational lower bound starting from the $\textsc{pds}$ Conjecture and including the heteroskedastic formulation of submatrix detection with different pairs $(\mP_{ij}, \mQ_{ij})$ at each entry is stated in Section \ref{sec:compbarriers}. We also give an alternative version requiring weaker bounds on $E_\mP$ and showing a slightly weaker formulation of the same computational lower bounds. In addition to computational lower bounds, we show information-theoretic lower bounds and give inefficient and polynomial time tests providing upper bounds at the barriers in submatrix detection. In Sections \ref{sec:polytests} and \ref{sec:searchtest}, we give general statements of these results given lower bounds on $E_\mP$ and $E_\mQ$ similar to the conditions in the theorem above.

\subsection{Universality Class with a Complete Phase Diagram}
\label{sec:completephasediagram}

\begin{figure*}[t!]
\centering
\begin{tikzpicture}[scale=0.45]
\tikzstyle{every node}=[font=\footnotesize]
\def\xmin{0}
\def\xmax{16}
\def\ymin{-1}
\def\ymax{11}

\draw[->] (\xmin,\ymin) -- (\xmax,\ymin) node[right] {$\beta$};
\draw[->] (\xmin,\ymin) -- (\xmin,\ymax) node[above] {$\alpha$};

\node at (15, -1) [below] {$1$};
\node at (7.5, -1) [below] {$\frac{1}{2}$};
\node at (0, 0) [left] {$0$};
\node at (0, 10) [left] {$2$};
\node at (0, 5) [left] {$1$};
\node at (0, 3.33) [left] {$\frac{2}{3}$};
\node at (10, -1) [below] {$\frac{2}{3}$};

\filldraw[fill=cyan!25, draw=blue] (0, 0) -- (7.5, 0) -- (10, 3.33) -- (0, 0);
\filldraw[fill=green!25, draw=green] (0, 0) -- (7.5, 0) -- (15, 10) -- (15, -1) -- (0, -1) -- (0, 0);
\filldraw[fill=gray!25, draw=gray] (0, 0) -- (10, 3.33) -- (15, 10) -- (0, 10) -- (0, 0);

\node at (4, -0.5) {$\SKL(\mP, \mQ) \asymp 1$};
\node at (12.2, 6)[rotate=54, anchor=south] {$\SKL(\mP, \mQ) \asymp \frac{n^2}{k^4}$};
\node at (4, 1.2)[rotate=20, anchor=south] {$\SKL(\mP, \mQ) \asymp \frac{1}{k}$};
\node at (6.5, 9) {information-theoretically impossible};
\node at (12, 2) {polynomial-time};
\node at (12, 1.25) {algorithms};
\node at (6.5, 1.25) {PC-hard};
\end{tikzpicture}

\caption{Computational and statistical barriers in $\pr{ssd}(n, k, \mP, \mQ)$ where $(\mP, \mQ)$ is in $\pr{uc-a}, \pr{uc-b}$ and $\pr{uc-c}$ with $k = \tilde{\Theta}(n^\beta)$ and $\SKL(\mP, \mQ) = \tilde{\Theta}(n^{-\alpha})$.}
\label{fig:spcaphasediagram}
\end{figure*}
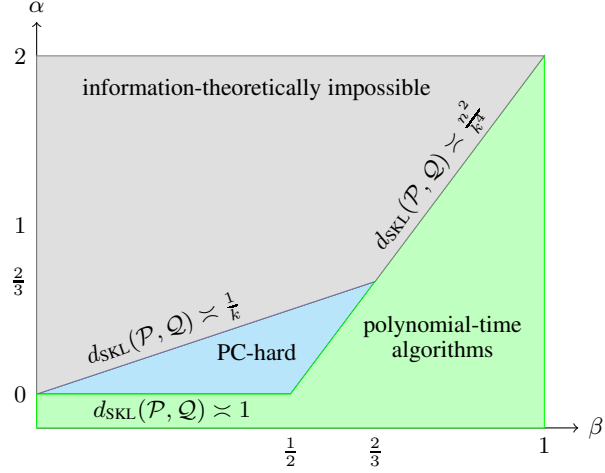

We now outline several assumptions on a computable pair $(\mP, \mQ)$ that allow our results to completely characterize the computational phase diagram for $\pr{ssd}(n, k, \mP, \mQ)$, by providing lower bounds on $E_\mP$ and $E_\mQ$. The first class we consider is a universality class that allows our average-case reduction to show lower bounds for submatrix detection.

\begin{definition}[Universality Class $\pr{uc-a}$]
Define $(\mP, \mQ)$ to be in the universality class $\textsc{uc-a}$ if $(\mP, \mQ)$ is computable and for any fixed $\epsilon \in (0, 1)$, it holds that
$$E_\mP(n^{\epsilon} \cdot \KL(\mP_n \| \mQ_n)) = \Omega(n^{\epsilon} \cdot \KL(\mP_n \| \mQ_n) \cdot \log n)$$
\end{definition}

The next universality class $\pr{uc-b}$ that we consider ensures the simple test statistics introduced in Sections \ref{sec:polytests} and \ref{sec:searchtest} show achievability at the computational and statistical barriers. This class is introduced as Assumption 2 in \cite{hajek2017information} and is weaker than sub-Gaussianity of the LLR.

\begin{definition}[Universality Class $\pr{uc-b}$]
Define $(\mP, \mQ)$ to be in the universality class $\textsc{uc-b}$ if $(\mP, \mQ)$ is computable and there is a constant $C \ge 1$ such that
\begin{align*}
\psi_{\mP}(\lambda) - \KL( \mP \| \mQ) \cdot \lambda &\le C \cdot \KL( \mP \| \mQ) \cdot \lambda^2 \quad \textnormal{for all } \lambda \in [-1, 0] \\
\psi_{\mQ}(\lambda) + \KL( \mQ \| \mP) \cdot \lambda &\le C \cdot \KL( \mQ \| \mP) \cdot \lambda^2 \quad \textnormal{for all } \lambda \in [-1, 1]
\end{align*}
\end{definition}

Our last universality class ensures that the information-theoretic lower bound that we show in Section \ref{sec:itlowerbounds} matches the upper bound from Section \ref{sec:searchtest}. Let $\SKL(\mP, \mQ) = \KL(\mP \| \mQ) + \KL(\mQ \| \mP)$ denote symmetric KL divergence.

\begin{definition}[Universality Class $\pr{uc-c}$]
Define $(\mP, \mQ)$ to be in the universality class $\textsc{uc-c}$ if $(\mP, \mQ)$ there is a constant $C' > 0$ such that $\chi^2(\mP \| \mQ) \le C' \cdot \SKL( \mP, \mQ)$.
\end{definition}

In Sections \ref{sec:compbarriers} and \ref{sec:statbarriers}, we specialize our general theorems on information-theoretic and computational upper and lower bounds to these three universality classes. This yields the following characterization of the computational phase diagram for $\textsc{ssd}(n, k, \mP, \mQ)$ when $(\mP, \mQ)$ is in $\pr{uc-a}$, $\pr{uc-b}$ and $\pr{uc-c}$. These regimes are depicted in Figure \ref{fig:spcaphasediagram}. Here, $\ll$ hides factors that are sub-polynomial $n$.

\begin{theorem}[Submatrix Detection Phase Diagram]
The regions of the computational phase diagram in $\textsc{ssd}(n, k, \mP, \mQ)$ are:
\begin{itemize}
\item (Statistically Impossible) if $(\mP, \mQ)$ is in $\pr{uc-c}$ then $\pr{ssd}$ is impossible if
$$\SKL(\mP, \mQ) \ll \frac{1}{k} \wedge \frac{n^2}{k^4}$$
\item ($\textsc{pc}$-Hard) if $(\mP, \mQ)$ is in $\pr{uc-a}$ then $\pr{ssd}$ is $\pr{pc}$-hard but possible if
$$\frac{1}{k} \wedge \frac{n^2}{k^4} \ll \SKL(\mP, \mQ) \ll \frac{n^2}{k^4} \wedge 1$$
\item (Polynomial Time Algorithms) if $(\mP, \mQ)$ is in $\pr{uc-b}$ then $\pr{ssd}$ can be solved in $\textnormal{poly}(n)$ time if
$$\frac{n^2}{k^4} \wedge 1 \ll \SKL(\mP, \mQ)$$
\end{itemize}
\end{theorem}

In Section \ref{sec:universalityclasses}, we discuss the three universality classes $\pr{uc-a}$, $\pr{uc-b}$ and $\pr{uc-c}$ and sub-classes of distributions that they contain. For example, these three classes contain the following pairs $(\mP, \mQ)$:
\begin{itemize}
\item Pairs $(\mP, \mQ)$ with sub-Gaussian LLR i.e. satisfying $L = \log \frac{d\mP}{d\mQ}(x)$ for $x \sim \mQ$ is sub-Gaussian. 
\item Pairs $(\mP, \mQ)$ with bounded LLR i.e. satisfying $L = \log \frac{d\mP}{d\mQ}(x)$ for $x \sim \mQ$ is bounded almost surely.
\item A wide variety of other pairs $(\mP, \mQ)$ from a common exponential family. The computations in Appendix B of \cite{hajek2017information} provide a simple method for determining if pairs $(\mP, \mQ)$ from an exponential family are in $\pr{uc-b}$. This same method can be applied to check membership in $\pr{uc-a}$ and $\pr{uc-c}$.
\end{itemize}
The class $\pr{uc-b}$ is discussed at length in Sections 2.1, 3 and Appendix B of \cite{hajek2017information}. In Section \ref{sec:universalityclasses}, we observe that these properties imply that the following three important computable pairs are in all three of the classes $\pr{uc-a}$, $\pr{uc-b}$ and $\pr{uc-c}$:
\begin{enumerate}
\item[($\mD_{\textsc{bc}}$)] $\mP = \mN(\mu, 1)$ and $\mQ = \mN(0, 1)$ where $\mu = n^{-\alpha}$ for some $\alpha > 0$, in which case $\pr{ssd}$ corresponds to Gaussian biclustering;
\item[($\mD_{\textsc{sp}}$)] $\mP = \text{Bern}(p)$ and $\mQ = \text{Bern}(q)$ where $p = cq = cn^{-\alpha}$ for some constant $c > 1$ and $\alpha > 0$, in which case the above diagonal entries of $\pr{ssd}$ are the adjacency matrix of an instance of sparse planted dense subgraph; and
\item[($\mD_{\textsc{gp}}$)] $\mP = \text{Bern}(p)$ and $\mQ = \text{Bern}(q)$ where $p = q + \Theta(n^{-\gamma})$ and $q = n^{-\alpha}$ for some constants $\gamma > \alpha > 0$, in which case the above diagonal entries of $\pr{ssd}$ are the adjacency matrix of an instance of general planted dense subgraph.
\end{enumerate}
The fact that these three examples are in our universality classes implies that our average-case reduction recovers and generalizes the planted clique lower bounds for Gaussian biclustering in \cite{ma2015computational, brennan2018reducibility} and for the sparse and general regimes of planted dense subgraph in \cite{hajek2015computational, brennan2018reducibility}. For the two graph problems, this is achieved by constructing a graph from the above diagonal terms of the matrix output by the reduction.

\section{Average-Case Reductions in Total Variation}
\label{sec:avgred}

\subsection{Reductions in Total Variation and the Computational Model}

As introduced in \cite{berthet2013complexity} and \cite{ma2015computational}, we give approximate reductions in total variation to show that lower bounds for one hypothesis testing problem imply lower bounds for another. These reductions yield an exact correspondence between the asymptotic Type I$+$II errors of the two problems. This is formalized in the following lemma, which is Lemma 3.1 from \cite{brennan2018reducibility} specialized to the case of simple vs. simple hypothesis testing. Its proof is short and follows from the definition of total variation.

\begin{lemma}[Lemma 3.1 in \cite{brennan2018reducibility}] \label{lem:3a}
Let $\mP_D$ and $\mP'_D$ be detection problems with hypotheses $H_0, H_1$ and  $H_0', H_1'$, respectively. Let $X$ be an instance of $\mathcal{P}_D$ and let $Y$ be an instance of $\mP'_D$. Suppose there is a polynomial time computable map $\mathcal{A}$ satisfying
$$\TV\left(\mL_{H_0}(\mathcal{A}(X)), \mL_{H_0'}(Y)\right) + \TV\left(\mL_{H_1}(\mathcal{A}(X)), \mL_{H_1'}(Y)\right) \le \delta$$
If there is a randomized polynomial time algorithm solving $\mP'_D$ with Type I$+$II error at most $\epsilon$, then there is a randomized polynomial time algorithm solving $\mP_D$ with Type I$+$II error at most $\epsilon + \delta$.
\end{lemma}

If $\delta = o(1)$, then given a blackbox solver $\mathcal{B}$ for $\mathcal{P}'_D$, the algorithm that applies $\mathcal{A}$ and then $\mathcal{B}$ solves $\mathcal{P}_D$ and requires only a single query to the blackbox. An algorithm that runs in randomized polynomial time refers to one that has access to $\text{poly}(n)$ independent random bits and must run in $\text{poly}(n)$ time where $n$ is the size of the instance of the problem. For clarity of exposition, in our reductions we assume that explicit expressions can be exactly computed and that we can sample a biased random bit $\text{Bern}(p)$ in polynomial time. We also assume that the oracles described in Definition \ref{def:computable} can be computed in $\text{poly}(n)$ time.

\subsection{Properties of Total Variation}

Throughout the proof of our main theorem, we will use the following well-known facts and inequalities concerning total variation distance.

\begin{fact} \label{tvfacts}
The distance $\TV$ satisfies the following properties:
\begin{enumerate}
\item (Triangle Inequality) Given three distributions $P, Q$ and $R$ on a measurable space $(\mathcal{X}, \mathcal{B})$, it follows that
$$\TV\left( P, Q \right) \le \TV\left( P, R \right) + \TV\left( Q, R \right)$$
\item (Data Processing) Let $P$ and $Q$ be distributions on a measurable space $(\mathcal{X}, \mathcal{B})$ and let $f : \mathcal{X} \to \mathcal{Y}$ be a Markov transition kernel. If $A \sim P$ and $B \sim Q$ then
$$\TV\left(\mL(f(A)), \mL(f(B))\right) \le \TV(P, Q)$$
\item (Tensorization) Let $P_1, P_2, \dots, P_n$ and $Q_1, Q_2, \dots, Q_n$ be distributions on a measurable space $(\mathcal{X}, \mathcal{B})$. Then
$$\TV\left( \prod_{i = 1}^n P_i, \prod_{i = 1}^n Q_i \right) \le \sum_{i = 1}^n \TV\left( P_i, Q_i \right)$$
\item (Conditioning on an Event) For any distribution $P$ on a measurable space $(\mathcal{X}, \mathcal{B})$ and event $A \in \mathcal{B}$, it holds that
$$\TV\left( P(\cdot | A), P \right) = 1 - P(A)$$
\item (Conditioning on a Random Variable) For any two pairs of random variables $(X, Y)$ and $(X', Y')$ each taking values in a measurable space $(\mathcal{X}, \mathcal{B})$, it holds that
$$\TV\left( \mL(X), \mL(X') \right) \le \TV\left( \mL(Y), \mL(Y') \right) + \bE_{y \sim Y} \left[ \TV\left( \mL(X | Y = y), \mL(X' | Y' = y) \right)\right]$$
where we define $\TV\left( \mL(X | Y = y), \mL(X' | Y' = y) \right) = 1$ for all $y \not \in \textnormal{supp}(Y')$.
\end{enumerate}
\end{fact}

Given an algorithm $\mathcal{A}$ and distribution $\mP$ on inputs, let $\mathcal{A}(\mP)$ denote the distribution of $\mathcal{A}(X)$ induced by $X \sim \mP$. If $\mathcal{A}$ has $k$ steps, let $\mathcal{A}_i$ denote the $i$th step of $\mathcal{A}$ and $\mathcal{A}_{i\text{-}j}$ denote the procedure formed by steps $i$ through $j$. Each time this notation is used, we clarify the intended initial and final variables when $\mathcal{A}_{i}$ and $\mathcal{A}_{i\text{-}j}$ are viewed as Markov kernels. The next lemma encapsulates the structure of all of our analyses of average-case reductions.

\begin{lemma} \label{lem:tvacc}
Let $\mathcal{A}$ be an algorithm that can be written as $\mathcal{A} = \mathcal{A}_m \circ \mathcal{A}_{m-1} \circ \cdots \circ \mathcal{A}_1$ for a sequence of steps $\mathcal{A}_1, \mathcal{A}_2, \dots, \mathcal{A}_m$. Suppose that the probability distributions $\mP_0, \mP_1, \dots, \mP_m$ are such that $\TV(\mathcal{A}_i(\mP_{i-1}), \mP_i) \le \epsilon_i$ for each $1 \le i \le m$. Then it follows that
$$\TV\left( \mathcal{A}(\mP_0), \mP_m \right) \le \sum_{i = 1}^m \epsilon_i$$
\end{lemma}

\begin{proof}
This follows from a simple induction on $m$. Note that the case when $m = 1$ follows by definition. Now observe that by the data-processing and triangle inequalities in Fact \ref{tvfacts}, we have that if $\mathcal{B} = \mathcal{A}_{m-1} \circ \mathcal{A}_{m-2} \circ \cdots \circ \mathcal{A}_1$ then
\begin{align*}
\TV\left( \mathcal{A}(\mP_0), \mP_m \right) &\le \TV\left( \mathcal{A}_m \circ \mathcal{B}(\mP_0), \mathcal{A}_m(\mP_{m - 1}) \right) + \TV\left(\mathcal{A}_m(\mP_{m - 1}), \mP_m \right) \\
&\le \TV\left( \mathcal{B}(\mP_0), \mP_{m - 1} \right) + \epsilon_m \\
&\le \sum_{i = 1}^m \epsilon_i
\end{align*}
where the last inequality follows from the induction hypothesis applied with $m - 1$ to $\mathcal{B}$. This completes the induction and proves the lemma.
\end{proof}

\section{Multivariate Rejection Kernels}
\label{sec:mrklarge}

\subsection{General MRK Algorithm and Analysis}

\begin{figure}[t!]
\begin{algbox}
\textbf{Algorithm} \textsc{mrk}$(B)$

\vspace{2mm}

\textit{Parameters}: Input $B \in \{0, 1\}$, parameter $n$, number of iterations $N$, dimension $\ell = \text{poly}(n)$, Bernoulli probabilities $0 < q < p \le 1$ and $\ell$ pairs of computable sequences of distributions $(\mP^i, \mQ^i)$ for $i \in \{ 1, 2, \dots, \ell \}$ over the measurable space $(X, \mathcal{B})$
\begin{enumerate}
\item Initialize $z = (z_1, z_2, \dots, z_\ell)$ arbitrarily in the support of $\mQ_n^1 \otimes \mQ_n^2 \otimes \cdots \otimes \mQ_n^\ell$
\item Until $z$ is set or $N$ iterations have elapsed:
\begin{enumerate}
\item[(1)] Form $z' = (z_1', z_2', \dots, z_\ell')$ where $z'_i \sim \mQ_n^i$ is sampled independently for each $i \in \{ 1, 2, \dots, \ell \}$ and compute the log-likelihood ratio
$$L_n(z') = \sum_{i = 1}^\ell \log \frac{d\mP^i_n}{d \mQ^i_n} (z_i')$$
\item[(2)] Proceed to the next iteration if it does not hold that
$$\log \left( \frac{1 - p}{1 - q} \right) \le L_n(z') \le \log \left( \frac{p}{q} \right)$$
\item[(3)] If $B = 0$, then set $z \gets z'$ with probability $1 - \frac{q}{p} \cdot \exp(L_n(z'))$
\item[(4)] If $B = 1$, then set $z \gets z'$ with probability $\frac{q}{p} \cdot \exp(L_n(z')) - \frac{q(1 - p)}{p(1 - q)}$
\end{enumerate}
\item Output $z$
\end{enumerate}
\vspace{1mm}
\end{algbox}
\caption{Multivariate rejection kernel algorithm with Bernoulli input $B \in \{0, 1\}$.}
\label{fig:rej-kernel}
\end{figure}

In this section, we introduce multivariate rejection kernels, which effectively perform an algorithmic change of measure simultaneously with a lift. More precisely, the map $\textsc{mrk}(B)$ sends a binary input $B \in \{0, 1\}$ to a higher dimensional space $X^\ell$ simultaneously satisfying two Markov transition properties:
\begin{enumerate}
\item if $B \sim \text{Bern}(p)$, then $\textsc{mrk}(B)$ is close to $\mP_n^1 \otimes \mP_n^2 \otimes \cdots \otimes \mP_n^\ell$ in total variation; and
\item if $B \sim \text{Bern}(q)$, then $\textsc{mrk}(B)$ is close to $\mQ_n^1 \otimes \mQ_n^2 \otimes \cdots \otimes \mQ_n^\ell$ in total variation.
\end{enumerate}
where $0 < q < p \le 1$ are fixed and $(\mP^i, \mQ^i)$ are pairs of computable sequences of distributions over a common measurable space $(X, \mathcal{B})$ for $i \in \{ 1, 2, \dots, \ell \}$. The maps $\textsc{mrk}$ will be a key part of our reduction to $\pr{ssd}$, lifting planted dense subgraph and planted clique instances to instances of submatrix detection with independent entries and the correct marginals.

These maps $\textsc{mrk}$ use rejection sampling to transform the distribution of the input from either one of two input Bernoulli distributions to two $\ell$-fold product distributions, without knowing which of the two distributions was the input. Multivariate rejection kernels are closely related to previous average-case reduction techniques in the literature. Univariate Markov transitions taking $\delta_1$ and $\text{Bern}(q)$ to either a fixed pair of Bernoulli distributions or a fixed pair of normal distributions were introduced in \cite{ma2015computational} and \cite{gao2017sparse}. More closely related to our map are the rejection kernel framework introduced in \cite{brennan2018reducibility} and the reduction from planted clique to planted dense subgraph in \cite{hajek2015computational}. The rejection kernels in \cite{brennan2018reducibility} used a similar rejection sampling scheme as we use here to map binary inputs to distributions on $\mathbb{R}$. In \cite{brennan2018reducibility}, it is assumed that both of the target distributions $f_X$ and $g_X$ have explicit density or mass functions and have sampling oracles. In contrast, our map $\textsc{mrk}$ only requires the Radon-Nikodym derivatives $\frac{d\mP_n^i}{d\mQ_n^i}$ exist and that there is a sampling oracle for the target noise distributions $\mQ_n^i$ rather than for $\mP_n^i$ as well, keeping the assumptions on $(\mP^i, \mQ^i)$ relatively minimal.

The reduction in \cite{hajek2015computational} relies on a multivariate Markov transition mapping $\text{Bern}(1)$ and $\text{Bern}(q)$ to distributions of the form $\text{Bern}(cQ)^{\otimes \ell}$ and $\text{Bern}(Q)^{\otimes \ell}$ for some constant $c > 1$. This is achieved by first performing a univariate map from $\text{Bern}(1)$ and $\text{Bern}(q)$ to $\text{Bin}(\ell, cQ)$ and $\text{Bin}(\ell, Q)$, respectively, and then observing that these counts are sufficient statistics for the target product distributions. As discussed in the survey \cite{wu2018statistical}, this approach extends naturally to other target distributions that have a common sufficient statistic with an explicit mass or density function. For example, the sum of entries is such a sufficient statistic for the family of distribution $\mN(\mu, 1)^{\otimes \ell}$ for $\mu \in \mathbb{R}$. On generalizing to any computable pairs $(\mP^i, \mQ^i)$, such a sufficient statistic common to $\mP_n^1 \otimes \mP_n^2 \otimes \cdots \otimes \mP_n^\ell$ and $\mQ_n^1 \otimes \mQ_n^2 \otimes \cdots \otimes \mQ_n^\ell$ may not exist. We remark that all of the univariate maps in \cite{ma2015computational}, \cite{gao2017sparse} and \cite{hajek2015computational} rely on sampling explicit mass or density functions that are linear combinations of the target distributions. This requires time at least the size of the support of the target distribution or of a sufficiently fine net of its support. In the multivariate case mapping to distributions on $X^\ell$, the size of this support grows exponentially in $\ell$ and this sampling scheme is not feasible in polynomial time. The rejection sampling scheme used here circumvents this issue. We now prove the main total variation guarantees of $\textsc{mrk}(B)$, generalizing the argument in Lemma 5.1 in \cite{brennan2018reducibility}.

\begin{lemma}[Multivariate Rejection Kernels] \label{lem:mrk}
Suppose that $n$ is a parameter, $0 < q < p \le 1$ and $N$ is a positive integer. Let $(\mP^i, \mQ^i)$ for $i \in \{ 1, 2, \dots, \ell \}$ be pairs of computable sequences of distributions over the measurable space $(X, \mathcal{B})$ and let
$$S = \left\{x \in X^\ell : \log \left( \frac{1 - p}{1 - q} \right) \le L_n(x) \le \log \left( \frac{p}{q} \right) \right\}$$
where $L_n(x) = \sum_{i = 1}^\ell \log \frac{d\mP^i_n}{d \mQ^i_n} (x_i)$ for each $x = (x_1, x_2, \dots, x_\ell) \in X^\ell$. Then there is a map $\textsc{mrk} : \{0, 1\} \to X^\ell$ that can be computed in $\textnormal{poly}(n, \ell, N)$ time satisfying that
\begin{align*}
\TV\left( \textsc{mrk}(\textnormal{Bern}(p)), \mP_n^1 \otimes \mP_n^2 \otimes \cdots \otimes \mP_n^\ell \right) &\le \Delta \\
\TV\left( \textsc{mrk}(\textnormal{Bern}(q)), \mQ_n^1 \otimes \mQ_n^2 \otimes \cdots \otimes \mQ_n^\ell \right) &\le \Delta
\end{align*}
where if $\mP^*_n = \mP_n^1 \otimes \mP_n^2 \otimes \cdots \otimes \mP_n^\ell$ and $\mQ^*_n = \mQ_n^1 \otimes \mQ_n^2 \otimes \cdots \otimes \mQ_n^\ell$, then $\Delta$ is given by
\begin{align*}
\Delta &= \frac{\bP_{x \sim \mQ^*_n}[x \not \in S] + \bP_{x \sim \mP^*_n}[x \not \in S]}{p - q} \\
&\quad \quad + \max \left\{ \left( \bP_{x \sim \mQ_n^*}[x \not \in S] + \frac{q}{p} \right)^N, \, \left( \frac{q}{p} \cdot \bP_{x \sim \mP_n^*}[x \not \in S] + \frac{p - 2pq + q^2}{p - pq} \right)^N \right\}
\end{align*}
\end{lemma}

\begin{proof}
Let $\textsc{mrk}$ be the $\text{poly}(n, \ell, N)$ time algorithm shown in Figure \ref{fig:rej-kernel}. For the sake of analysis, consider continuing to iterate Step 2 even after $z$ is set for the first time for a total of $N$ iterations. Let $A_i^0$ and $A_i^1$ be the events that $z$ is set in the $i$th iteration of Step 2 when $B = 0$ and $B = 1$, respectively. For $B = 0$, let $B_i^0 = (A_1^0)^C \cap (A_2^0)^C \cap \cdots \cap (A^0_{i - 1})^C \cap A_i^0$ be the event that $z$ is set for the first time in the $i$th iteration of Step 2. Let $C^0 = A_1^0 \cup A_2^0 \cup \cdots \cup A_N^0$ be the event that $z$ is set in some iteration of Step 2. Define $B_i^1$ and $C^1$ analogously. Let $z_0$ be the initialization of $z$ in Step 1.

Now let $Z_0 \sim \mD_0 = \mL(\textsc{mrk}(0))$ and $Z_1 \sim \mD_1 = \mL(\textsc{mrk}(1))$. We have that $\mL(Z_0|B_i^0) = \mL(Z_0|A_i^0)$ and $\mL(Z_1|B_i^1) = \mL(Z_1|A_i^1)$ since $A_i^t$ is independent of $A_1^t, A_2^t, \dots, A_{i-1}^t$ for each $t \in \{0, 1\}$ and the sample $z'$ chosen in the $i$th iteration of Step 2. Observe that independence between Steps 2.1, 2.3 and 2.4 ensures that
\begin{align*}
\bP\left[A_i^0\right] &= \bE_{x \sim \mQ^*_n}\left[ \left( 1 - \frac{q}{p} \cdot \exp(L_n(x)) \right) \cdot \mathbf{1}_{S}(x) \right] \\
&= \bP_{x \sim \mQ^*_n}[x \in S] - \frac{q}{p} \cdot \bP_{x \sim \mP^*_n}[x \in S] \\
\bP\left[A_i^1\right] &= \bE_{x \sim \mQ^*_n}\left[ \left( \frac{q}{p} \cdot \exp(L_n(x)) - \frac{q(1 - p)}{p(1 - q)} \right) \cdot \mathbf{1}_{S}(x) \right] \\
&= \frac{q}{p} \cdot \bP_{x \sim \mP^*_n}[x \in S] - \frac{q(1 - p)}{p(1 - q)} \cdot \bP_{x \sim \mQ^*_n}[x \in S] 
\end{align*}
since $\frac{d\mP^*_n}{d\mQ^*_n}(x) = \exp\left( L_n(x)\right)$. By the independence of the $A_i^0$, we have that
\begin{align*}
1 - \bP\left[ C^0 \right] &= \prod_{i = 1}^N \left( 1 - \bP\left[A_i^0\right] \right) = \left( 1 - \bP_{x \sim \mQ^*_n}[x \in S] + \frac{q}{p} \cdot \bP_{x \sim \mP^*_n}[x \in S] \right)^N \\
&\le \left( \bP_{x \sim \mQ^*_n}[x \not \in S] + \frac{q}{p} \right)^N
\end{align*}
Similarly, we have that
\begin{align*}
1 - \bP\left[ C^1 \right] &= \prod_{i = 1}^N \left( 1 - \bP\left[A_i^1\right] \right) = \left( 1 - \frac{q}{p} \cdot \bP_{x \sim \mP^*_n}[x \in S] + \frac{q(1 - p)}{p(1 - q)} \cdot \bP_{x \sim \mQ^*_n}[x \in S]  \right)^N \\
&\le \left( \frac{q}{p} \cdot \left( 1 - \bP_{x \sim \mP^*_n}[x \in S] \right) + 1 - \frac{q}{p} + \frac{q(1 - p)}{p(1 - q)} \right)^N \\
&\le \left( \frac{q}{p} \cdot \bP_{x \sim \mP^*_n}[x \not \in S] + \frac{p - 2pq + q^2}{p - pq} \right)^N
\end{align*}
We now have that $\mL(Z_0|A_i^0)$ and $\mL(Z_1|A_i^1)$ are each absolutely continuous with respect to $\mQ^*_n$ with Radon-Nikodym derivatives
\begin{align*}
\frac{d\mL(Z_0|B_i^0)}{d\mQ^*_n} (x) = \frac{d\mL(Z_0|A_i^0)}{d\mQ^*_n} (x) &= \bP\left[A_i^0\right]^{-1} \cdot \left( 1 - \frac{q}{p} \cdot \exp(L_n(x)) \right) \cdot \mathbf{1}_S(x) \\
&= \frac{p - q \cdot \frac{d \mP_n^*}{d\mQ_n^*}(x)}{p \cdot \bP_{x \sim \mQ^*_n}[x \in S] - q \cdot \bP_{x \sim \mP^*_n}[x \in S]} \cdot \mathbf{1}_S(x) \\
\frac{d\mL(Z_1|B_i^1)}{d\mQ^*_n} (x) = \frac{d\mL(Z_1|A_i^1)}{d\mQ^*_n} (x) &= \bP\left[A_i^1\right]^{-1} \cdot \left( \frac{q}{p} \cdot \exp(L_n(z')) - \frac{q(1 - p)}{p(1 - q)} \right) \cdot \mathbf{1}_S(x) \\
&= \frac{(1 - q) \cdot \frac{d \mP_n^*}{d\mQ_n^*}(x) - (1 - p)}{(1 - q) \cdot \bP_{x \sim \mP^*_n}[x \in S] - (1 - p) \cdot \bP_{x \sim \mQ^*_n}[x \in S]} \cdot \mathbf{1}_S(x)
\end{align*}
Now observe that since the conditional laws $\mL(Z_0|B_i^0)$ are all identical, we have that
$$\frac{d\mD_0}{d\mQ^*_n} (x) = \bP\left[C^0 \right] \cdot \frac{d\mL(Z_0|B_1^0)}{d\mQ^*_n} (x) + \left( 1 - \bP\left[C^0 \right] \right) \cdot \mathbf{1}_{z_0}(x)$$
Therefore it follows that
\begin{align*}
\TV\left( \mD_0, \mL(Z_0|B_1^0) \right) &= \frac{1}{2} \cdot \bE_{x \sim \mQ^*_n} \left[\left| \frac{d\mD_0}{d\mQ^*_n} (x) - \frac{d\mL(Z_0|B_1^0)}{d\mQ^*_n} (x) \right| \right] \\
&= \frac{1}{2} \left( 1 - \mP\left[ C^0 \right] \right) \cdot \bE_{x \sim \mQ^*_n} \left[ \left| \mathbf{1}_{z_0}(x) - \frac{d\mL(Z_0|B_1^0)}{d\mQ^*_n} (x) \right| \right] \\
&\le \frac{1}{2} \left( 1 - \mP\left[ C^0 \right] \right) \cdot \bE_{x \sim \mQ^*_n} \left[ \mathbf{1}_{z_0}(x) + \frac{d\mL(Z_0|B_1^0)}{d\mQ^*_n} (x) \right] \\
&= 1 - \mP\left[ C^0 \right]
\end{align*}
by the triangle inequality. A symmetric argument implies that $\TV\left( \mD_1, \mL(Z_1|B_1^1) \right) \le 1 - \mP\left[ C^1 \right]$. Now observe that since $\frac{d \mP_n^*}{d\mQ_n^*}(x) \le \frac{p}{q}$ for $x \in S$, we have that
\begin{align*}
&\bE_{x \sim \mQ^*_n} \left[\left| \frac{d\mL(Z_0|B_1^0)}{d\mQ^*_n} (x) - \frac{p - q \cdot \frac{d \mP_n^*}{d\mQ_n^*}(x)}{p - q} \right| \right] \\ 
&\quad \quad = \left|\frac{1}{p \cdot \bP_{x \sim \mQ^*_n}[x \in S] - q \cdot \bP_{x \sim \mP^*_n}[x \in S]} - \frac{1}{p - q} \right| \cdot \bE_{x \sim \mQ^*_n} \left[\left( p - q \cdot \frac{d \mP_n^*}{d\mQ_n^*}(x) \right) \cdot \mathbf{1}_S(x) \right] \\
&\quad \quad \quad \quad + \bE_{x \sim \mQ^*_n} \left[ \left| \frac{p - q \cdot \frac{d \mP_n^*}{d\mQ_n^*}(x)}{p - q} \right| \cdot \mathbf{1}_{S^C}(x) \right] \\
&\quad \quad \le \left| 1 - \frac{p \cdot \bP_{x \sim \mQ^*_n}[x \in S] - q \cdot \bP_{x \sim \mP^*_n}[x \in S]}{p - q} \right| + \bE_{x \sim \mQ^*_n} \left[ \frac{p + q \cdot \frac{d \mP_n^*}{d\mQ_n^*}(x)}{p - q} \cdot \mathbf{1}_{S^C}(x) \right] \\
&\quad \quad = \left| \frac{p \cdot \bP_{x \sim \mQ^*_n}[x \not \in S] - q \cdot \bP_{x \sim \mP^*_n}[x \not \in S]}{p - q} \right| + \frac{p \cdot \bP_{x \sim \mQ^*_n}[x \not \in S] + q \cdot \bP_{x \sim \mP^*_n}[x \not \in S]}{p - q} \\
&\quad \quad \le \frac{2 \cdot \bP_{x \sim \mQ^*_n}[x \not \in S] + 2 \cdot \bP_{x \sim \mP^*_n}[x \not \in S]}{p - q}
\end{align*}
By a symmetric computation, we have that
$$\bE_{x \sim \mQ^*_n} \left[\left| \frac{d\mL(Z_1|B_1^1)}{d\mQ^*_n} (x) - \frac{(1 - q) \cdot \frac{d \mP_n^*}{d\mQ_n^*}(x) - (1 - p)}{p - q} \right| \right] \le \frac{2 \cdot \bP_{x \sim \mQ^*_n}[x \not \in S] + 2 \cdot \bP_{x \sim \mP^*_n}[x \not \in S]}{p - q}$$
Now observe that
\begin{align*}
\frac{d \mP_n^*}{d\mQ_n^*}(x) &= p \cdot \frac{(1 - q) \cdot \frac{d \mP_n^*}{d\mQ_n^*}(x) - (1 - p)}{p - q} + (1 - p) \cdot \frac{p - q \cdot \frac{d \mP_n^*}{d\mQ_n^*}(x)}{p - q} \\
1 &= q \cdot \frac{(1 - q) \cdot \frac{d \mP_n^*}{d\mQ_n^*}(x) - (1 - p)}{p - q} + (1 - q) \cdot \frac{p - q \cdot \frac{d \mP_n^*}{d\mQ_n^*}(x)}{p - q}
\end{align*}
Now observe that
\begin{align*}
\TV\left( \textsc{mrk}(\text{Bern}(p)), \mP^*_n \right) &= \TV\left( p \cdot \mL(Z_1 | B_1^1) + (1 - p) \cdot \mL(Z_0 | B_1^0), \mP^*_n \right) \\
&\quad \quad +  \TV\left( p \cdot \mL(Z_1 | B_1^1) + (1 - p) \cdot \mL(Z_0 | B_1^0), \textsc{mrk}(\text{Bern}(p)) \right) \\
&\le \frac{p}{2} \cdot \bE_{x \sim \mQ^*_n} \left[\left| \frac{d\mL(Z_1|B_1^1)}{d\mQ^*_n} (x) - \frac{p - q \cdot \frac{d \mP_n^*}{d\mQ_n^*}(x)}{p - q} \right| \right] \\
&\quad \quad + \frac{(1 - p)}{2} \cdot \bE_{x \sim \mQ^*_n} \left[\left| \frac{d\mL(Z_0|B_1^0)}{d\mQ^*_n} (x) - \frac{p - q \cdot \frac{d \mP_n^*}{d\mQ_n^*}(x)}{p - q} \right| \right] \\
&\quad \quad + p \cdot \TV\left( \mD_1, \mL(Z_1|B_1^1) \right) + (1 - p) \cdot \TV\left( \mD_0, \mL(Z_0|B_1^0) \right) \\
&\le \Delta
\end{align*}
A symmetric argument shows $\TV\left( \textsc{mrk}(\text{Bern}(q)), \mQ^*_n \right) \le \Delta$, completing the proof of the lemma.
\end{proof}

\subsection{Homogeneous MRK and Log-Likelihood Ratio LDPs}

We now show that when all of the pairs $(\mP^i, \mQ^i) = (\mP, \mQ)$ are the same and their LLR satisfies LDPs with respect to each of $\mQ$ and $\mP$, then the total variation error $\Delta$ is small in the lemma above. This will be the form of the lemma that we will apply to show computational lower bounds for $\pr{ssd}$.

\begin{lemma} \label{lem:mrkupperbound}
Let $(\mP, \mQ)$ be a computable pair over $(X, \mathcal{B})$. Define $n, \ell, p$ and $q$ as in Lemma \ref{lem:mrk}. Let $\tau_+, \tau_- \ge \ell^{-1} \cdot \log \left( 4(p - q)^{-1} \right)$ and suppose that $(\mP, \mQ)$ satisfies
$$\log\left( \frac{1 - q}{1 - p} \right) < -\ell \cdot \KL(\mQ \| \mP) \le \ell \cdot \KL(\mP \| \mQ) < \log \left( \frac{p}{q} \right)$$
and the LLR between $(\mP, \mQ)$ satisfies the LDPs
$$ E_\mP\left( \ell^{-1} \cdot \log \left( \frac{p}{q} \right) \right) \ge \tau_+ \quad \textnormal{and} \quad E_\mQ\left( \ell^{-1} \cdot \log \left( \frac{1 - p}{1 - q} \right) \right) \ge \tau_-$$
where the second inequality is only necessary if $p \neq 1$. Then there is a map $\textsc{mrk} : \{0, 1\} \to X^\ell$ that can be computed in $\textnormal{poly}(n, \ell, pq^{-1}(p - q)^{-1}, \min(\tau_+, \tau_-))$ time satisfying that
\begin{align*}
\TV\left( \textsc{mrk}(\textnormal{Bern}(p)), \mP_n^{\otimes \ell} \right) &\le \frac{3(e^{-\ell \tau_+} + e^{-\ell \tau_-})}{p - q} \\
\TV\left( \textsc{mrk}(\textnormal{Bern}(q)), \mQ_n^{\otimes \ell} \right) &\le \frac{3(e^{-\ell \tau_+} + e^{-\ell \tau_-})}{p - q}
\end{align*}
\end{lemma}

\begin{proof}
Let $c_+ = \log \left( \frac{p}{q} \right) > 0$ and $c_- = \log \left( \frac{1 - p}{1 - q} \right) < 0$. Observe that
\begin{align*}
E_\mQ\left( c_+ \ell^{-1} \right) &= E_\mP\left( c_+ \ell^{-1} \right) + c_+ \ell^{-1} > \tau_+ \\
E_\mP\left( c_- \ell^{-1} \right) &= E_\mQ\left( c_- \ell^{-1} \right) - c_- \ell^{-1} > \tau_-
\end{align*}
where the second set of inequalities holds if $p \neq 1$. Now consider applying Lemma \ref{lem:mrk} with number of iterations
$$N = \left\lceil \frac{\ell \cdot \min(\tau_+, \tau_-)}{\log \left( 1 - \frac{q(p - q)}{2p} \right)^{-1}}\right\rceil = O\left( \frac{p}{q(p - q)} \cdot \ell \cdot \min(\tau_+, \tau_-) \right)$$
since $\log(1 - x)^{-1} \ge x$ for $x = \frac{q(p - q)}{2p}$. Thus it suffices to bound $\Delta$ given that the LLR between $(\mP, \mQ)$ satisfies the LDPs above. For now consider the case where $p \neq 1$. Let $L_n(x) = \sum_{i = 1}^\ell \log \frac{d\mP_n}{d \mQ_n} (x_i)$ for each $x \in X^\ell$. Now consider $\lambda_1$ with objective value approaching the supremum in the optimization $E_\mP\left(c_+ \ell^{-1} \right) = \sup_{\lambda \in \mathbb{R}} c_+ \ell^{-1} \lambda - \psi_\mP(\lambda)$. Since the derivative of the objective at $\lambda = 0$ is $c_+ \ell^{-1} - \KL(\mP \| \mQ) > 0$, we can take $\lambda_1 \ge 0$. Analogously, we can take $\lambda_2 \le 0$ approaching the supremum in the optimization $E_\mP(c_- \ell^{-1})$. Applying a Chernoff bound now yields that
\begin{align*}
\bP_{x \sim \mP_n^{\otimes \ell}} \left[ x \not \in  S \right] &= \bP_{x \sim \mP_n^{\otimes \ell}}\left[ \exp\left( \lambda_1 \cdot L_n(x) \right) > \left( \frac{p}{q} \right)^{\lambda_1} \right] + \bP_{x \sim \mP_n^{\otimes \ell}}\left[ \exp\left( \lambda_2 \cdot L_n(x) \right) > \left( \frac{1 - p}{1 - q} \right)^{\lambda_2} \right] \\
&\le \left( \frac{p}{q} \right)^{-\lambda_1} \cdot \bE_{X \sim \mP_n^{\otimes \ell}}\left[ \exp\left( \lambda_1 \cdot L_n(x) \right) \right] + \left( \frac{1 - q}{1 - p} \right)^{-\lambda_2} \cdot \bE_{X \sim \mP_n^{\otimes \ell}}\left[ \exp\left( -\lambda_2 \cdot L_n(x) \right) \right] \\
&= \exp\left( - c_+ \lambda_1 + \ell \cdot \psi_\mP(\lambda_1) \right) + \exp\left( - c_- \lambda_2 + \ell \cdot \psi_\mP(\lambda_2) \right)  \\
&\to \exp\left(-\ell \cdot E_\mP(c_+ \ell^{-1}) \right) + \exp\left(-\ell \cdot E_\mP(c_- \ell^{-1}) \right) \\
&\le e^{-\ell \tau_+} + e^{-\ell \tau_-}
\end{align*}
where the limit holds on taking $\lambda_1$ and $\lambda_2$ to approach the two suprema. Note that if $p = 1$, then the second probability above trivially satisfies that
$$\bP_{x \sim \mP_n^{\otimes \ell}}\left[ L_n(x) < \frac{1 - p}{1 - q} \right] = 0 < e^{-\ell \tau_-}$$
and the bound on $\bP_{x \sim \mP_n^{\otimes \ell}} \left[ x \not \in  S \right]$ above still holds. By a symmetric argument, we also have that
$$\bP_{x \sim \mQ_n^{\otimes \ell}} \left[ x \not \in S \right] \le e^{-\ell \tau_+} + e^{-\ell \tau_-}$$
Now observe that since $e^{-\ell \tau_+}, e^{-\ell \tau_-} \le \frac{1}{4} (p - q)$, we have that
$$\frac{q}{p} \cdot \bP_{x \sim \mP_n^{\otimes \ell}} \left[ x \not \in S \right] + \frac{p - 2pq + q^2}{p - pq} \le \frac{q}{p} \cdot \left( \frac{1}{2} (p - q) \right) + 1 - \frac{q(p - q)}{p(1 - q)} \le 1 - \frac{q(p - q)}{2p(1 - q)} \le 1 - \frac{q(p - q)}{2p}$$
Similarly, we have
$$\bP_{x \sim \mQ_n^{\otimes \ell}} \left[ x \not \in S \right] + \frac{q}{p} \le \frac{1}{2} (p - q) + 1 - \frac{p - q}{p} \le 1 - \frac{q(p - q)}{2p}$$
Therefore it follows that
\begin{align*}
&\max \left\{ \left( \bP_{x \sim \mQ_n^*}[x \not \in S] + \frac{q}{p} \right)^N, \, \left( \frac{q}{p} \cdot \bP_{x \sim \mP_n^*}[x \not \in S] + \frac{p - 2pq + q^2}{p - pq} \right)^N \right\} \\
&\quad \quad \le \left(1 - \frac{q(p - q)}{2p} \right)^N \le \max\left( e^{-\ell \tau_+}, e^{-\ell \tau_-} \right)
\end{align*}
and hence that $\Delta \le \frac{3(e^{-\ell \tau_+} + e^{-\ell \tau_-})}{p - q}$, completing the proof of the lemma.
\end{proof}

\subsection{Entrywise Reductions Fail to Show Tight Computational Lower Bounds}
\label{sec:entrywise}

A conceptually simpler idea for an average-case reduction would avoid multivariate rejection kernels and instead begin by mapping to a submatrix problem with a specific pair $(\mP^*, \mQ^*)$ and then apply a univariate entry-wise map from $(\mP^*, \mQ^*)$ to $(\mP, \mQ)$. For example, if $(\mP^*, \mQ^*)$ were Bernoulli random variables, this would correspond to mapping to a submatrix variant of planted dense subgraph and then to $\pr{ssd}(n, k, \mP, \mQ)$. We remark that there is a simple barrier to making this approach produce tight computational lower bounds.

As discussed in the summary of our results in Section \ref{sec:summary}, the relevant signal in the computational phase diagram for $\pr{ssd}$ is $\KL(\mP \| \mQ)$. For this approach to lead to a reduction, there would have to be a univariate map $\phi$ such that $\phi(\mP^*) \sim \mP$ and $\phi(\mQ^*) \sim \mQ$ as long as $\KL(\mP^* \| \mQ^*) \asymp \KL(\mP \| \mQ)$. However, looking at the family of pairs of Bernoulli random variables shows that this is not the case. If $(\mP^*, \mQ^*) = (\text{Bern}(n^{-\alpha}), \text{Bern}(2n^{-\alpha}))$ then we have that
$$\KL(\mP^* \| \mQ^*) = -n^{-\alpha} \log 2 - (1 - n^{-\alpha}) \log\left( \frac{1 - 2n^{-\alpha}}{1 - n^{-\alpha}} \right) = \Theta(n^{-\alpha}) \quad \text{and} \quad \TV(\mQ^* \| \mP^*) = n^{-\alpha}$$
Furthermore, let $(\mP, \mQ) = (\text{Bern}(1/2), \text{Bern}(1/2 + n^{-\alpha/2}))$ and note that
$$\KL(\mP \| \mQ) = - \frac{1}{2} \log(1 + 2n^{-\alpha/2}) - \frac{1}{2}\log(1 - 2n^{-\alpha/2}) = \Theta(n^{-\alpha}) \quad \text{and} \quad \TV(\mQ \| \mP) = n^{-\alpha/2}$$
By the data-processing applied to $\TV$, it is impossible for such a map $\phi$ to exist even though these two pairs have KL divergences on the same order.

\section{Average-Case Reduction to Submatrix Detection}
\label{sec:avgredsub}

\begin{figure}[t!]
\begin{algbox}
\textbf{Algorithm} \textsc{To-Submatrix}

\vspace{1mm}

\textit{Inputs}: Graph $G \in \mG_n$, parameters $0 < q < p \le 1$, intermediate dimension $N$, expansion factor $\ell$, number of iterations $N_{\textnormal{it}}$, subgraph size $k$, pairs of computable sequences of distributions $(\mP_{ij}, \mQ_{ij})$ for $1 \le i, j \le N\ell$ over the measurable space $(X, \mathcal{B})$
\begin{enumerate}
\item Apply $\textsc{Graph-Clone}$ to $G$ with edge probabilities $P = p$ and $Q = 1 - \sqrt{(1 - p)(1 - q)} + \mathbf{1}_{\{p = 1\}} \left( \sqrt{q} - 1 \right)$ and $t = 2$ clones to obtain $(G_1, G_2)$
\item Sample $s_1 \sim \text{Bin}(n, p)$, $s_2 \sim \text{Bin}(N, Q)$ and a set $S \subseteq [N]$ with $|S| = n$ uniformly at random. Sample $T_1 \subseteq S$ and $T_2 \subseteq [N] \backslash S$ with $|T_1| = s_1$ and $|T_2| = \max\{ s_2 - s_1, 0 \}$ uniformly at random. Now form the matrix $M^1 \in \{0, 1\}^{N \times N}$ where
$$M^1_{ij} = \left\{ \begin{array}{ll} \mathbf{1}_{\{\pi(i), \pi(j)\} \in E(G_1)} & \text{if } i < j \text{ and } i, j \in S \\ \mathbf{1}_{\{\pi(i), \pi(j)\} \in E(G_2)} & \text{if } i > j \text{ and } i, j \in S \\ \mathbf{1}_{\{ i \in T_1 \}} & \text{if } i = j \text{ and } i, j \in S \\ \mathbf{1}_{\{i \in T_2\}} & \text{if } i = j \text{ and } i, j \not \in S \\ \sim_{\text{i.i.d.}} \text{Bern}(Q) & \text{if } i \neq j \text{ and } i \not \in S \text{ or } j \not \in S \end{array} \right.$$
where $\pi : S \to [n]$ is a bijection chosen uniformly at random
\item Let $\tau$ be a random permutation of $[N\ell]$ and form the matrix $M^2 \in X^{N\ell \times N\ell}$ by setting
$$\left(M^2_{ij} : s\ell + 1 \le \tau^{-1}(i) \le (s + 1)\ell \text{ and } t \ell + 1 \le \tau^{-1}(j) \le (t + 1)\ell\right) = \textsc{mrk}_{st}\left(M^1_{(s+1)(t+1)}\right)$$
for each $0 \le s, t < N$, where $\textsc{mrk}_{st}$ is the multivariate rejection kernel with $N_{\textnormal{it}}$ iterations sending $\textnormal{Bern}(p)$ and $\textnormal{Bern}(Q)$ to the random set of $\ell^2$ pairs $(\mP_{ij}^n, \mQ_{ij}^n)$ satisfying that $s\ell + 1 \le \tau^{-1}(i) \le (s + 1)\ell$ and $t \ell + 1 \le \tau^{-1}(j) \le (t + 1)\ell$
\item Output the matrix $M^2$
\end{enumerate}
\vspace{1mm}

\end{algbox}
\caption{Reduction $\textsc{To-Submatrix}$ for showing universal computational lower bounds for submatrix problems based on the hardness of planted clique or planted dense subgraph.}
\label{fig:tosubmatrix}
\end{figure}

In this section, we give a polynomial time reduction $\textsc{To-Submatrix}$ to $\pr{ssd}$ detection with pairs of planted and noise distributions that are computable and satisfy an LDP of their LLR. This reduction is shown in Figure \ref{fig:tosubmatrix} and its total variation guarantees are stated in Theorem \ref{thm:tosubmatrix} below. The reduction begins by cloning the adjacency matrix of a planted dense subgraph instance to produce two independent samples of the above diagonal portions of their adjacency matrix. These are then embedded as two halves of a principal minor in a larger matrix in Step 2. This random embedding hides the previously missing diagonal terms in total variation. Analyzing this step is one of our technical contributions. The resulting Bernoulli submatrix problem is then lifted using $\textsc{mrk}$ maps to an instance of submatrix detection.

Our reduction implies the computational lower bounds for Gaussian biclustering and planted dense subgraph in \cite{ma2015computational}, \cite{hajek2015computational} and \cite{brennan2018reducibility}. The step of embedding as a principal minor in Step 2 circumvents the arguments in \cite{hajek2015computational} and \cite{brennan2018reducibility} showing that the total variation error introduced by missing diagonal entries is small. It also simplifies the reductions $\textsc{PC-Lifting}$ and $\textsc{Distributional-Lifting}$ in \cite{brennan2018reducibility}. Our reduction extends the Gaussian biclustering lower bounds in \cite{ma2015computational} and \cite{brennan2018reducibility} to the cases of symmetric indices.

In Section \ref{subsec:graphcloning}, we describe and analyze the subroutine $\textsc{Graph-Clone}$ in Step 1 which transforms one instance of planted dense subgraph into many independent instances, preserving both $H_0$ and $H_1$, at a minor loss in the parameters $p$ and $q$. In Section \ref{subsec:plantingdiagonals}, we prove several lemmas to analyze Step 2 of $\textsc{To-Submatrix}$ and in Section \ref{subsec:proofofthm}, we complete the proof of Theorem \ref{thm:tosubmatrix}. In the next section, we apply Theorem \ref{thm:tosubmatrix} to deduce computational lower bounds for various forms of submatrix detection. Note that the reduction $\textsc{To-Submatrix}$ handles a more general setup where the pairs $(\mP_{ij}, \mQ_{ij})$ are permitted to differ from entry to entry. The additional submatrix problems this implies hardness for are discussed in Section \ref{sec:compbarriers}. $\textsc{To-Submatrix}$ also begins with an instance of planted dense subgraph, providing a more general class of reductions than just from planted clique.

Let $\mathcal{M}_{d}( (\mQ_{ij})_{1 \le i, j \le d})$ denote the distribution on $d \times d$ matrices $M$ with independent entries satisfying that $M_{ij} \sim \mQ_{ij}$ for each $1 \le i, j \le d$. Let $\mathcal{M}_{d}( (\mP_{ij}, \mQ_{ij})_{1 \le i, j \le d}, k)$ denote the distribution on $d \times d$ matrices $M$ generated as follows: choose a $k$-subset $S \subseteq [d]$ uniformly at random and then independently sample $M_{ij} \sim \mP_{ij}$ for each $i, j \in S$ and $M_{ij} \sim \mQ_{ij}$ otherwise.

\begin{theorem}[Reduction to Submatrix Detection] \label{thm:tosubmatrix}
Let $n$, $k \le n$, $N$, $N_{\textnormal{it}}$, $\ell$ and $0 < q < p \le 1$ be parameters such that $q = n^{-O(1)}$,
$$N \ge \left(\frac{p}{Q} + \epsilon \right)n, \quad k \le \frac{Q \epsilon n}{2} \quad \text{and} \quad \frac{k^2}{N} \le \min \left\{ \frac{Q}{1 - Q}, \frac{1 - Q}{Q} \right\}$$
where $\epsilon > 0$ and $Q = 1 - \sqrt{(1 - p)(1 - q)} + \mathbf{1}_{\{p = 1\}} \left( \sqrt{q} - 1 \right)$. Let $(\mP_{ij}, \mQ_{ij})$ for $1 \le i, j \le N\ell$ be pairs of computable sequences of distributions indexed by $n$ over the measurable space $(X, \mathcal{B})$. For each pair of $\ell$-subsets $U, V \subseteq [N\ell]$, let $\Delta_{U, V}$ be the total variation upper bound as defined in Lemma \ref{lem:mrk} for the multivariate rejection kernel $\textsc{mrk}_{U, V}$ with $N_{\textnormal{it}}$ iterations sending $\textnormal{Bern}(p)$ and $\textnormal{Bern}(Q)$ to the $\ell^2$ pairs $(\mP_{ij}^n, \mQ_{ij}^n)$ where $i \in U$ and $j \in V$. Let $\Delta = \max_{U, V} \Delta_{U, V}$ be the maximum such upper bound. The algorithm $\mathcal{A} = \textsc{To-Submatrix}$ in Figure \ref{fig:tosubmatrix} runs in $\textnormal{poly}(N, \ell)$ time and satisfies that
\begin{align*}
\TV\left( \mathcal{A}(\mG(n, k, p, q)), \, \mathcal{M}_{N\ell}( (\mP_{ij}, \mQ_{ij})_{1 \le i, j \le N\ell}, k\ell) \right) &\le N^2 \cdot \Delta + 4 \cdot \exp\left( - \frac{Q \epsilon^2 n^2}{32N} \right) \\
&\quad \quad + \sqrt{\frac{k^2(1 - Q)}{2QN}} + \sqrt{\frac{k^2Q}{2N(1 - Q)}} \\
\TV\left( \mathcal{A}(\mG(n, q)), \, \mathcal{M}_{N\ell}( (\mQ_{ij})_{1 \le i, j \le N\ell}) \right) &\le N^2 \cdot \Delta + 4 \cdot \exp\left( - \frac{Q \epsilon^2 n^2}{32N} \right)
\end{align*}
\end{theorem}

Before proceeding to the proofs in this section, we first establish some additional notation. When all of the pairs $(\mP_{ij}, \mQ_{ij})$ are the same $(\mP, \mQ)$, then denote $\mathcal{M}_{d}( (\mP_{ij}, \mQ_{ij})_{1 \le i, j \le d}, k)$ by $\mathcal{M}_{d}(\mP, \mQ, k)$. Given a $k$-subset $S \subseteq [d]$, let $\mathcal{M}_{d}( (\mP_{ij}, \mQ_{ij})_{1 \le i, j \le d}, S)$ denote $\mathcal{M}_{d}( (\mP_{ij}, \mQ_{ij})_{1 \le i, j \le d}, k)$ conditioned on the selection of $S$ as the planted index set. Define $\mathcal{M}_{d}(\mP, \mQ, S)$ analogously. Similarly, let $\mathcal{V}_{N}\left( \mathcal{P}, \mathcal{Q}, S \right)$ denotes the distribution of vectors $V \in X^N$ with independent entries and $V_i \sim \mathcal{P}$ if $i \in S$ and $V_i \sim \mathcal{Q}$ if $i \not \in S$. Let $\mathcal{V}_{N}\left( \mathcal{P}, \mathcal{Q}, k \right)$ denote the mixture of $\mathcal{V}_{N}\left( \mathcal{P}, \mathcal{Q}, S \right)$ induced by choosing $S$ uniformly at random from all $k$-subsets of $[N]$. Let $\mG(n, S, p, q)$ denote the distribution of planted dense subgraph instances from $\mG(n, k, p, q)$ conditioned on the subgraph being planted on the vertex set $S$ where $|S| = k$. Given an algorithm $\mathcal{A}$ with $k$ steps, let $\mathcal{A}_i$ denote the $i$th step of $\mathcal{A}$ and $\mathcal{A}_{i\text{-}j}$ denote the procedure formed by steps $i$ through $j$. Each time this notation is used, we clarify the intended initial and final variables when $\mathcal{A}_{i}$ and $\mathcal{A}_{i\text{-}j}$ are viewed as Markov kernels.

\subsection{Graph Cloning}
\label{subsec:graphcloning}

\begin{figure}[t!]
\begin{algbox}
\textbf{Algorithm} \textsc{Graph-Clone}

\vspace{1mm}

\textit{Inputs}: Graph $G \in \mG_n$, the number of copies $t$, parameters $0 < q < p \le 1$ and $0 < Q < P \le 1$ satisfying $\frac{1 - p}{1 - q} \le \left( \frac{1 - P}{1 - Q} \right)^t$ and $\left( \frac{P}{Q} \right)^t \le \frac{p}{q}$

\begin{enumerate}
\item Generate $x^{ij} \in \{0, 1\}^t$ for each $1 \le i < j \le n$ such that:
\begin{itemize}
\item If $\{i, j \} \in E(G)$, sample $x^{ij}$ from the distribution on $\{0, 1\}^t$ with
$$\bP[x^{ij} = v] = \frac{1}{p - q} \left[ (1 - q) \cdot P^{|v|_1} (1 - P)^{t - |v|_1} - (1 - p) \cdot Q^{|v|_1} (1 - Q)^{t - |v|_1} \right]$$
\item If $\{i, j \} \not \in E(G)$, sample $x^{ij}$ from the distribution on $\{0, 1\}^t$ with
$$\bP[x^{ij} = v] = \frac{1}{p - q} \left[ p \cdot Q^{|v|_1} (1 - Q)^{t - |v|_1} - q \cdot P^{|v|_1} (1 - P)^{t - |v|_1} \right]$$
\end{itemize}
\item Output the graphs $(G_1, G_2, \dots, G_t)$ where $\{i, j\} \in E(G_k)$ if and only if $x^{ij}_k = 1$
\end{enumerate}
\vspace{1mm}

\end{algbox}
\caption{Subroutine $\textsc{Graph-Clone}$ for producing independent samples from planted graph problems.}
\label{fig:tosubmatrix}
\end{figure}

We begin with the subroutine $\textsc{Graph-Clone}$ which produces several independent samples from a planted subgraph problems given a single sample. This procedure is a simple generalization of $\textsc{PDS-Cloning}$ in Section 10 of \cite{brennan2018reducibility} and can be viewed as an exact variant of a simple multivariate rejection kernel to products of Bernoulli random variables.

\begin{lemma}[Graph Cloning] \label{lem:cloning}
Let $t \in \mathbb{N}$, $0 < q < p \le 1$ and $0 < Q < P \le 1$ satisfy that
$$\frac{1 - p}{1 - q} \le \left( \frac{1 - P}{1 - Q} \right)^t \quad \text{and} \quad \left( \frac{P}{Q} \right)^t \le \frac{p}{q}$$
Then the algorithm $\mathcal{A} = \textsc{Graph-Clone}$ runs in $\textnormal{poly}(t, n)$ time and satisfies that for each $S \subseteq [n]$,
$$\mathcal{A}\left( \mG(n, q) \right) \sim \mG(n, Q)^{\otimes t} \quad \text{and} \quad \mathcal{A}\left( \mG(n, S, p, q) \right) \sim \mG(n, S, P, Q)^{\otimes t}$$
\end{lemma}

\begin{proof}
Let $R_0, R_1 : \{0, 1\}^t \to \mathbb{R}$ be given by
\begin{align*}
R_0(v) &= \frac{1}{p - q} \left[ p \cdot Q^{|v|_1} (1 - Q)^{t - |v|_1} - q \cdot P^{|v|_1} (1 - P)^{t - |v|_1} \right] \\
R_1(v) &= \frac{1}{p - q} \left[ (1 - q) \cdot P^{|v|_1} (1 - P)^{t - |v|_1} - (1 - p) \cdot Q^{|v|_1} (1 - Q)^{t - |v|_1} \right]
\end{align*}
Now observe that for each $v \in \{0, 1\}^t$, the fact that $P > Q$ implies that
$$\frac{1 - p}{1 - q} \le \left( \frac{1 - P}{1 - Q} \right)^t \le \frac{P^{|v|_1} (1 - P)^{t - |v|_1}}{Q^{|v|_1} (1 - Q)^{t - |v|_1}} \le \left( \frac{P}{Q} \right)^t \le \frac{p}{q}$$
which implies that $R_0(v) \ge 0$ and $R_1(v) \ge 0$ for each $v \in \{0, 1\}^t$. Furthermore, we have that
$$\sum_{v \in \{0, 1\}^t} R_0(v) = \sum_{v \in \{0, 1\}^t} R_1(v) = 1$$
which implies that $R_0$ and $R_1$ are well-defined probability mass functions. Also observe that
\begin{align*}
(1 - p) \cdot R_0(v) + p \cdot R_1(v) &= P^{|v|_1} (1 - P)^{t - |v|_1} \\
(1 - q) \cdot R_0(v) + q \cdot R_1(v) &= Q^{|v|_1} (1 - Q)^{t - |v|_1}
\end{align*}
Therefore it follows that if $\mathbf{1}_{\{i, j\} \in E(G)} \sim \text{Bern}(p)$, then $x^{ij} \sim \text{Bern}(P)^{\otimes t}$ and if $\mathbf{1}_{\{i, j\} \in E(G)} \sim \text{Bern}(q)$, then $x^{ij} \sim \text{Bern}(Q)^{\otimes t}$. Since the edge indicators $\mathbf{1}_{\{i, j\} \in E(G)}$ are independent in each of $\mG(n, q)$ and $\mG(n, S, p, q)$, this implies that $(G_1, G_2, \dots, G_t) \sim \mG(n, Q)^{\otimes t}$ if $G \sim \mG(n, q)$ and $(G_1, G_2, \dots, G_t) \sim \mG(n, P)^{\otimes t}$ if $G \sim \mG(n, k, p, q)$, completing the proof of the lemma.
\end{proof}

\subsection{Planting Diagonals by Embedding as a Principal Minor}
\label{subsec:plantingdiagonals}

The next two lemmas are a key technical component in the analysis of $\textsc{To-Submatrix}$. Specifically, they are crucial to showing the correctness of Step 2 in $\textsc{To-Submatrix}$, which plants missing diagonal entries while randomly embedding entries derived from the adjacency matrix of the input instance as a principal minor into a larger matrix to hide the planted entries in total variation. We remark that the applications of Cauchy-Schwarz reducing the proof of the second lemma to bounding $\chi^2$ divergences are unlikely to be tight. However, the resulting bounds are sufficient for our purposes. We also remark that Lemma \ref{lem:plantingdiagonals} can be proven by directly bounding sums of differences of binomial coefficients. Instead, our approach yields more elegant computations, as carried out in Lemma \ref{lem:hidingentries}, and can be generalized to bound sums of random variables beyond binomial distributions.

\begin{lemma} \label{lem:hidingentries}
Suppose that $\mathcal{P}$ and $\mathcal{Q}$ are probability distributions on a measurable space $(X,\mathcal{B})$ where $\mathcal{P}$ is absolutely continuous with respect to $\mathcal{Q}$. Then for any positive integers $k$ and $m$ with $k^2 \cdot \chi^2( \mathcal{P} \| \mathcal{Q}) \le m$, we have that
$$\chi^2\left( \mathcal{V}_m(\mathcal{P}, \mathcal{Q}, k) \, \| \, \mathcal{Q}^{\otimes m} \right) \le \frac{2k^2 \cdot \chi^2( \mathcal{P} \| \mathcal{Q})}{m}$$
\end{lemma}

\begin{proof}
Let $f : X \to [0, \infty)$ be the Radon-Nikodym derivative $f = \frac{d\mathcal{P}}{d\mathcal{Q}}$. Note that $\mathcal{V}_m(\mathcal{P}, \mathcal{Q}, k)$ is absolutely continuous with respect to $\mathcal{Q}^{\otimes m}$ with Radon-Nikodym derivative
$$\frac{d\mathcal{V}_m(\mathcal{P}, \mathcal{Q}, k)}{d\mathcal{Q}^{\otimes m}} (x) = \bE_{S \sim \mathcal{U}_{k, m}} \left[ \prod_{i \in S} f(x_i) \right]$$
for each $x \in X^m$ where $\mathcal{U}_{k, m}$ is the uniform distribution on $k$-subsets of $[m]$. Now note that by Fubini's theorem
\begin{align*}
\chi^2\left( \mathcal{V}_m(\mathcal{P}, \mathcal{Q}, k) \, \| \, \mathcal{Q}^{\otimes m} \right) + 1 &= \bE_{x \sim \mathcal{Q}^{\otimes m}} \left[ \left( \frac{d\mathcal{V}_m(\mathcal{P}, \mathcal{Q}, k)}{d\mathcal{Q}^{\otimes m}} (x) \right)^2 \right] \\
&= \bE_{x \sim \mathcal{Q}^{\otimes m}} \left[ \bE_{S \sim \mathcal{U}_{k, m}} \left[ \prod_{i \in S} f(x_i) \right] \cdot \bE_{T \sim \mathcal{U}_{k, m}} \left[ \prod_{i \in S} f(x_i) \right] \right] \\
&= \bE_{S, T \sim \mathcal{U}_{k, m}} \left[ \bE_{x \sim \mathcal{Q}^{\otimes m}} \left[ \left( \prod_{i \in S} f(x_i) \right) \left( \prod_{i \in T} f(x_i) \right) \right] \right] \\
&= \bE_{S, T \sim \mathcal{U}_{k, m}} \left[ \prod_{i \in S \cap T} \bE_{x_i \sim \mathcal{Q}}\left[f(x_i)^2\right] \prod_{i \in S \backslash T} \bE_{x_i \sim \mathcal{Q}}\left[f(x_i) \right] \prod_{i \in T \backslash S} \bE_{x_i \sim \mathcal{Q}}\left[f(x_i)\right] \right] \\
&= \bE_{S, T \sim \mathcal{U}_{k, m}} \left[ (1 + \chi^2(\mathcal{P} \| \mathcal{Q}))^{|S \cap T|} \right]
\end{align*}
where the last equality holds since $\bE_{x_i \sim \mathcal{Q}}[f(x_i)] = 1$ and $1 + \chi^2(\mathcal{P} \| \mathcal{Q}) = \bE_{x_i \sim \mathcal{Q}}[f(x_i)^2]$. We now apply an argument in \cite{addario2010combinatorial} to bound this last quantity. Observe that $|S \cap T| \sim \text{Hypergeometric}(m, k, k)$ and is identically distributed to $|[k] \cap S| = \sum_{i = 1}^k \mathbf{1}_{\{i \in S\}}$. As shown in Section 3.2 of \cite{joag1983negative}, the variables $\mathbf{1}_{\{i \in S\}}$ are negatively associated which implies that
\begin{align*}
\bE_{S, T \sim \mathcal{U}_{k, m}} \left[ (1 + \chi^2(\mathcal{P} \| \mathcal{Q}))^{|S \cap T|} \right] &= \bE_{S \sim \mathcal{U}_{k, m}} \left[ \prod_{i = 1}^k (1 + \chi^2(\mathcal{P} \| \mathcal{Q}))^{\mathbf{1}_{\{i \in S\}}} \right] \\
&\le \prod_{i = 1}^k \bE_{S \sim \mathcal{U}_{k, m}} \left[  (1 + \chi^2(\mathcal{P} \| \mathcal{Q}))^{k \cdot \mathbf{1}_{\{i \in S\}}} \right] \\
&= \left( \frac{k}{m} (1 + \chi^2(\mathcal{P} \| \mathcal{Q})) + 1 - \frac{k}{m} \right)^k \\
&= \left( 1 + \frac{k}{m} \cdot \chi^2( \mathcal{P} \| \mathcal{Q}) \right)^k \\
&\le \exp\left( \frac{k^2 \cdot \chi^2( \mathcal{P} \| \mathcal{Q})}{m} \right) \le 1 + \frac{2k^2 \cdot \chi^2( \mathcal{P} \| \mathcal{Q})}{m}
\end{align*}
if $k^2 \cdot \chi^2( \mathcal{P} \| \mathcal{Q}) \le m$. This completes the proof of the lemma.
\end{proof}

We now apply this lemma to bound the total variation between support sizes of the matrix diagonals produced in $\textsc{To-Submatrix}$ and the target matrix distributions.

\begin{lemma}[Planting Diagonals] \label{lem:plantingdiagonals}
Suppose that $0 < Q < P \le 1$ and $N \ge \left( \frac{P}{Q} + \epsilon \right) n$ where $\epsilon > 0$. Let $k \le n$ satisfy that
$$k \le \frac{Q \epsilon n}{2} \quad \text{and} \quad \frac{k^2}{N} \le \min \left\{ \frac{Q}{1 - Q}, \frac{1 - Q}{Q} \right\}$$
Let $t_1 \sim \textnormal{Bin}(k, P)$, $t_2 \sim \textnormal{Bin}(n - k, P)$ and $t_3 \sim \textnormal{Bin}(N, Q)$ be independent and set $t_4 = \max \{ t_3 - t_1 - t_2, 0 \}$. Then it holds that
\begin{align*}
\TV\left( \mL(t_1, t_2 + t_4), \, \textnormal{Bin}(k, P) \otimes \textnormal{Bin}(N - k, Q) \right) &\le 4 \cdot \exp\left( - \frac{Q \epsilon^2 n^2}{32N} \right) + \sqrt{\frac{k^2(1 - Q)}{2NQ}} + \sqrt{\frac{k^2Q}{2N(1 - Q)}} \\
\TV\left( \mL(t_1 + t_2 + t_4), \, \textnormal{Bin}(N, Q) \right) &\le 4 \cdot \exp\left( - \frac{Q \epsilon^2 n^2}{32N} \right)
\end{align*}
\end{lemma}

\begin{proof}
First consider applying Lemma \ref{lem:hidingentries} to $\mathcal{P} = \delta_0$ and $\mathcal{Q} = \text{Bern}(Q)$. Note that since $Q \in (0, 1)$, this choice of $\mathcal{P}$ is absolutely continuous with respect to $\mathcal{Q}$. We have that if $K$ and $M$ are such that $K^2 Q \le M (1 - Q)$ then
$$\chi^2\left( \mathcal{V}_M( \delta_0, \text{Bern}(Q), K) \| \text{Bern}(Q)^{\otimes M} \right) \le \frac{2K^2 Q}{M (1 - Q)}$$
Now by Cauchy-Schwarz and the data-processing property in Fact \ref{tvfacts} on taking the sum of the entries of the vectors, we have that
\begin{align*}
\TV\left( \text{Bin}(M - K, Q), \text{Bin}(M, Q) \right) &\le \TV\left( \mathcal{V}_M( \delta_0, \text{Bern}(Q), K), \, \text{Bern}(Q)^{\otimes M} \right) \\
&\le \frac{1}{2} \sqrt{\chi^2\left( \mathcal{V}_M( \delta_0, \text{Bern}(Q), K) \| \text{Bern}(Q)^{\otimes M} \right)} \\
&\le \sqrt{\frac{K^2 Q}{2M (1 - Q)}}
\end{align*}
Now apply Lemma \ref{lem:hidingentries} to $\mathcal{P} = \delta_1$ and $\mathcal{Q} = \text{Bern}(Q)$. By the same argument, if $K^2 (1 - Q) \le M Q$, then we have that
$$\TV\left( K + \text{Bin}(M - K, Q), \text{Bin}(M, Q) \right) \le \sqrt{\frac{K^2 (1 - Q)}{2MQ}}$$
Combining these two inequalities and applying the triangle inequality in Fact \ref{tvfacts} yields that
\begin{align*}
\TV\left( K' + \text{Bin}(M - K, Q), \text{Bin}(M, Q) \right) &\le \TV\left( K' + \text{Bin}(M - K, Q), \text{Bin}(M - K + K', Q) \right) \\
&\quad \quad + \TV\left( \text{Bin}(M - K + K', Q), \text{Bin}(M, Q) \right) \\
&\le \sqrt{\frac{K'^2 (1 - Q)}{2(M - K + K') Q}} + \sqrt{\frac{(K - K')^2 Q}{2M (1 - Q)}}
\end{align*}
as long as $K' \le K$ and it holds that
$$K'^2 \le \frac{(M - K + K') Q}{1 - Q} \quad \text{and} \quad (K - K')^2 \le \frac{M(1 - Q)}{Q}$$
Note that both of these inequalities are satisfied if $K^2/M \le \min \left\{ Q^{-1} (1 - Q), Q(1 - Q)^{-1} \right\}$. Standard Chernoff bounds imply that the binomial distribution satisfies the following concentration inequalities
\begin{align*}
\bP\left[\text{Bin}(M, r) > (1 + \delta) rM \right] &\le \exp\left( - \frac{\delta^2}{2 + \delta} \cdot rM \right) \\
\bP\left[\text{Bin}(M, r) < (1 - \delta) rM \right] &\le \exp\left( - \frac{\delta^2}{2} \cdot rM \right)
\end{align*}
for $\delta > 0$. These inequalities can be derived by standard Chernoff bounds. Now observe that if $t_3 \ge Q N - \frac{Q \epsilon n}{2} + \frac{m}{2}$ and $t_2 \le P(n - k) + \frac{Q \epsilon n}{2} - \frac{m}{2}$ hold, then $t_3 \ge m + t_2$ since $QN \ge (P + Q\epsilon)n$. Therefore we have that
\begin{align*}
\bP\left[ t_4 \neq t_3 - t_1 - t_2 | t_1 = m  \right] &= \bP[t_3 < m + t_2] \\
&\le \bP\left[ t_3 < Q N - \frac{Q \epsilon n}{2} +\frac{m}{2} \right] + \bP\left[ t_2 > P(n-k) + \frac{Q \epsilon n}{2} - \frac{m}{2} \right] \\
&\le \bP\left[ t_3 < Q N - \frac{Q \epsilon n}{4} \right] + \bP\left[ t_2 > P(n-k) + \frac{Q \epsilon n}{4} \right] \\
&\le \exp\left( - \frac{1}{2} \left( \frac{\epsilon n}{4N} \right)^2 \cdot QN \right) \\
&\quad \quad + \exp\left( - \frac{1}{2 +\frac{Q \epsilon n}{4P(n - k)}} \cdot  \left( \frac{Q \epsilon n}{4P(n - k)} \right)^2 \cdot P(n - k) \right) \\
&\le \exp\left( - \frac{Q\epsilon^2 n^2}{32N} \right) + \exp\left( - \frac{1}{8P(n - k) + Q\epsilon n} \cdot \frac{Q^2 \epsilon^2 n^2}{4} \right) \\
&\le 2 \cdot \exp\left( - \frac{Q \epsilon^2 n^2}{32N} \right)
\end{align*}
The first inequality above is a union bound, the second inequality holds since $2m \le 2k \le Q \epsilon n$ and the last inequality holds since $8P(n - k) + Q \epsilon n \le 8QN$. Marginalizing this bound over $t_1$ yields that
$$\bP[t_4 \neq t_3 - t_1 - t_2] = \bE_{m \sim \mL(t_1)} \bP\left[ t_4 \neq t_3 - t_1 - t_2 | t_1 = m \right] \le 2 \cdot \exp\left( - \frac{Q \epsilon^2 n^2}{32N} \right)$$
Now by the conditioning property in Fact \ref{tvfacts}, we have
\begin{align*}
\TV\left( \mL(t_1 + t_2 + t_4), \mL(t_3 | t_1 + t_2 + t_4 = t_3) \right) &\le \bP\left[ t_1 + t_2 + t_4 \neq t_3 \right] \\
\TV\left( \mL(t_3), \mL(t_3 | t_1 + t_2 + t_4 = t_3) \right) &\le \bP\left[ t_1 + t_2 + t_4 \neq t_3 \right]
\end{align*}
Since $t_3 \sim \text{Bin}(N, Q)$, the triangle inequality in Fact \ref{tvfacts} implies that
$$\TV\left( \mL(t_1 + t_2 + t_4), \text{Bin}(N, Q) \right) \le 2 \cdot \bP\left[ t_1 + t_2 + t_4 \neq t_3 \right] \le 4 \cdot \exp\left( - \frac{Q \epsilon^2 n^2}{32N} \right)$$
which proves the second inequality in the lemma. Now observe that by the conditioning property in Fact \ref{tvfacts}, we have that
\begin{align*}
\TV\left( \mL(t_2 + t_4 | t_1 = m), \mL(t_2 + t_4 | t_1 = m, t_1 + t_2 + t_4 = t_3) \right) &\le \bP\left[ t_4 \neq t_3 - t_1 - t_2 | t_1 = m  \right] \\
\TV\left( \mL(t_3), \mL(m + t_2 + t_4 | t_1 = m, t_1 + t_2 + t_4 = t_3) \right) &\le \bP\left[ t_4 \neq t_3 - t_1 - t_2 | t_1 = m  \right]
\end{align*}
Note that $t_3 \sim \text{Bin}(N, Q)$ is independent of $t_1$ and thus $\mL(t_3) = \mL(t_3 | t_1 = m)$. Applying the inequality derived above using Lemma \ref{lem:hidingentries} with $M = N$, $K = k$ and $K' = m$ yields that
\begin{align*}
\TV\left( \mL(t_3), m + \text{Bin}(N - k, Q) \right) &\le \sqrt{\frac{m^2 (1 - Q)}{2(N - k + m) Q}} + \sqrt{\frac{(k - m)^2 Q}{2N (1 - Q)}} \\
&\le \sqrt{\frac{k^2(1 - Q)}{2NQ}} + \sqrt{\frac{k^2Q}{2N(1 - Q)}}
\end{align*}
as long as $k^2/N \le \min \left\{ Q^{-1} (1 - Q), Q(1 - Q)^{-1} \right\}$. Applying the triangle inequality twice now yields that
\begin{align*}
&\TV\left( \mL(t_2 + t_4 | t_1 = m), \text{Bin}(N - k, Q) \right) \\
&\quad \quad = \TV\left( \mL(m + t_2 + t_4 | t_1 = m), m + \text{Bin}(N - k, Q) \right) \\
&\quad \quad \le 2 \cdot \bP\left[ t_4 \neq t_3 - t_1 - t_2 | t_1 = m \right] + \sqrt{\frac{k^2(1 - Q)}{2NQ}} + \sqrt{\frac{k^2Q}{2N(1 - Q)}} \\
&\quad \quad \le 4 \cdot \exp\left( - \frac{Q \epsilon^2 n^2}{32N} \right) + \sqrt{\frac{k^2(1 - Q)}{2NQ}} + \sqrt{\frac{k^2Q}{2N(1 - Q)}}
\end{align*}
Now note that by the conditioning on a random variable property in Fact \ref{tvfacts}, we have that
$$\TV\left( \mL(t_1, t_2 + t_4), \, \textnormal{Bin}(k, P) \otimes \textnormal{Bin}(N - k, Q) \right) \le \bE_{m \sim \textnormal{Bin}(k, P)} \TV\left( \mL(t_2 + t_4 | t_1 = m), \text{Bin}(N - k, Q) \right)$$
Combining this with the inequality derived above completes the proof of the lemma.
\end{proof}

\subsection{Proof of Theorem \ref{thm:tosubmatrix}}
\label{subsec:proofofthm}

We now combine the lemmas in the previous two sections and Lemma \ref{lem:mrk} to prove Theorem \ref{thm:tosubmatrix}. First consider the case where $G \sim \mG(n, R, p, q)$ where $R \subseteq [n]$ satisfies $|R| = k$. In the first part of the proof of this proposition, let $M^1 = \mathcal{A}_{\text{1-2}}(G)$ be the matrix $M^1$ after Steps 1 and 2 in $\mathcal{A}$. First observe by AM-GM that
$$\sqrt{pq} \le \frac{p + q}{2} = 1 - \frac{(1 - p) + (1 - q)}{2} \le 1 - \sqrt{(1 - p)(1 - q)}$$
Let $Q = 1 - \sqrt{(1 - p)(1 - q)} + \mathbf{1}_{\{p = 1\}} \left( \sqrt{q} - 1 \right)$. If $p \neq 1$, then it follows that $P = p > Q$, $\frac{1 - p}{1 - q} = \left( \frac{1 - P}{1 - Q} \right)^2$ and the inequality above rearranges to $\left( \frac{P}{Q} \right)^2 \le \frac{p}{q}$. If $p = 1$, then  $Q = \sqrt{q}$, the inequality $\frac{1 - p}{1 - q} \le \left( \frac{1 - P}{1 - Q} \right)^2$ holds trivially and $\left( \frac{P}{Q} \right)^2 = \frac{p}{q}$. Applying Lemma \ref{lem:cloning} with $t = 2$ therefore yields that $(G_1, G_2) \sim \mG(n, R, p, Q)^{\otimes 2}$. Let $U = \pi^{-1}(R)$ be the subset of $[N]$ that the dense subgraph vertices are mapped to in Step 2, on choosing $S$ and $\pi$. Let $R' \subseteq [N]$ be a fixed subset with $|R'| = k$. Observe that the matrix $M^1$ in Step 2 conditioned on $U = R'$ has independent off-diagonal entries satisfying $M^1_{ij} \sim \text{Bern}(p)$ if $i, j \in R'$ and $M^1_{ij} \sim \text{Bern}(Q)$ otherwise, matching the off-diagonal distribution of $\mathcal{M}_{N}\left( \text{Bern}(p), \text{Bern}\left(Q \right), R' \right)$. Furthermore these entries are independent of the diagonal entries of $M^1$. Thus the tensorization property in Fact \ref{tvfacts} implies that
\begin{align*}
&\TV\left( \mL(M^1 | U = R'), \mathcal{M}_{N}\left( \text{Bern}(p), \text{Bern}\left(Q \right), R' \right) \right) \\
&\quad \quad = \TV\left( \mL(\text{diag}(M^1) | U = R'), \mathcal{V}_{N}\left( \text{Bern}(p), \text{Bern}\left(Q \right), R' \right)\right)
\end{align*}
Fix some subset $S' \subseteq [N]$ with $|S'| = n$. Now observe that conditioned on $S = S'$, the entries $M^1_{ii}$ with $i \in S'$ are i.i.d. distributed as $\text{Bern}(p)$ since the number of $i \in S'$ with $M^1_{ii} = 1$ is $s_1 \sim \text{Bin}(n, p)$ and $(M^1_{ii} : i \in S')$ is exchangeable. Therefore, conditioned on $U = R'$ and $S = S'$, the entries of $\text{diag}(M^1)$ are distributed as $(M^1_{ii} : i \in S') \sim \text{Bern}(p)^{\otimes n}$ and $(M^1_{ii} : i \not \in S')$ is exchangeable with support of size $|T_2| = \max\{ s_2 - s_1, 0\}$ where $s_1$ is the size of the support of $(M^1_{ii} : i \in S')$ and $s_2 \sim \text{Bin}(N, Q)$ is sampled independently. Since $S$ is chosen uniformly at random, conditioned on $U = R'$, the elements of $S \backslash R'$ are a uniformly at random chosen subset of $[N]\backslash R'$ of size $n - k$. Thus relaxing the conditioning to only $U = R'$ yields that the entries of $\text{diag}(M^1)$ are distributed as $(M^1_{ii} : i \in R') \sim \text{Bern}(p)^{\otimes k}$ and $(M^1_{ii} : i \not \in R')$ is exchangeable with support of size $t_2 + |T_2| = t_2 + \max\{ s_2 - t_1 - t_2, 0\}$ where $t_1$ is the size of the support of $(M^1_{ii} : i \not \in R')$ and $t_2 \sim \text{Bin}(n - k, p)$ is sampled independently.

Note that the distributions of $\mL(\text{diag}(M^1) | U = R')$ and $\mathcal{V}_{N}\left( \text{Bern}(p), \text{Bern}\left(Q \right), R' \right)$ restricted to the indices in $R'$ and $[N] \backslash R'$ are each exchangeable. Therefore conditioning on the pair of support sizes within $R'$ and $[N]\backslash R'$ and applying the conditioning property in Fact \ref{tvfacts} yields that
\begin{align*}
&\TV\left( \mL(\text{diag}(M^1) | U = R'), \mathcal{V}_{N}\left( \text{Bern}(p), \text{Bern}\left(Q \right), R' \right)\right) \\
&\quad \quad = \TV\left( \mL(t_1,  t_2 + \max\{ s_2 - t_1 - t_2, 0\}), \text{Bin}(k, p) \otimes \text{Bin}\left(N - k, Q\right) \right) \\
&\quad \quad \le 4 \cdot \exp\left( - \frac{Q \epsilon^2 n^2}{32N} \right) + \sqrt{\frac{k^2(1 - Q)}{2QN}} + \sqrt{\frac{k^2Q}{2N(1 - Q)}}
\end{align*}
by Lemma \ref{lem:plantingdiagonals}. Applying the conditioning property in Fact \ref{tvfacts} to conditioning on $R$ and $U = R'$ now yields that
\begin{align*}
&\TV\left( \mL(\mathcal{A}_{\text{1-2}}(\mG(n, k, p, q))), \mathcal{M}_{N}\left( \text{Bern}(p), \text{Bern}\left(Q \right), k \right) \right) \\
&\quad \quad \le \bE_{R \sim \mathcal{U}_{k, N}}\bE_U \left[ \TV\left( \mL(\text{diag}(M^1) | U), \mathcal{V}_{N}\left( \text{Bern}(p), \text{Bern}\left(Q \right), U \right)\right) \right] \\
&\quad \quad \le 4 \cdot \exp\left( - \frac{Q \epsilon^2 n^2}{32N} \right) + \sqrt{\frac{k^2(1 - Q)}{2QN}} + \sqrt{\frac{k^2Q}{2N(1 - Q)}}
\end{align*}
where $\mathcal{U}_{k, N}$ is the uniform distribution on the $k$-subsets of $[N]$. Now let $\mathcal{A}_3$ denote Step 3 of $\mathcal{A}$ with input $M^1$ and output $M^2$. Let $M^1 \sim \mathcal{M}_{N}\left( \text{Bern}(p), \text{Bern}\left(Q \right), R' \right)$ and $M^2 = \mathcal{A}_3(M^1)$. Consider also conditioning on the permutation $\tau = \tau'$ where $\tau'$ is a fixed permutation of $[N\ell]$. Let $U_s = \tau'(\{s\ell + 1, s\ell + 2, \dots, (s + 1)\ell \})$ for each $0 \le s < N$ and note that $\textsc{mrk}_{st} = \textsc{mrk}_{U_s, U_t}$. Now applying Lemma \ref{lem:mrk} to the rejection kernels $\textsc{mrk}_{st}$ yields that
\begin{align*}
\TV\left(\textsc{mrk}_{st}(\text{Bern}(p)), \, \Motimes_{(i, j) \in U_s \times U_t} \mP_{ij}^n \right) &\le \Delta \quad \text{and} \\
\TV\left(\textsc{mrk}_{st}\left(\text{Bern}\left( Q \right) \right), \, \Motimes_{(i, j) \in U_s \times U_t} \mQ_{ij}^n \right) &\le \Delta
\end{align*}
Now let $V = \bigcup_{i \in R'} U_{i-1}$ be the set of indices of $[N\ell]$ that $R'$ is mapped to and let $M'$ be sampled as $M' \sim \mathcal{M}_{N\ell}( (\mP_{ij}, \mQ_{ij})_{1 \le i, j \le N\ell}, V)$. The tensorization property in Fact \ref{tvfacts} now yields that
\begin{align*}
&\TV\left( \mL(M^2 | \tau = \tau'), \, \mathcal{M}_{N\ell}( (\mP_{ij}, \mQ_{ij})_{1 \le i, j \le N\ell}, V) \right) \\
&\quad \quad \le \sum_{i, j = 1}^N \TV\left( \mL(M^2_{ab} : a \in U_{i-1}, b \in U_{j-1}), \mL( M'_{ab} : a \in U_{i-1}, b \in U_{j-1}) \right) \\
&\quad \quad = \sum_{(i, j) \in R'^2} \TV\left(\textsc{rk}_{(i-1)(j-1)}(\text{Bern}(p)), \, \Motimes_{(a, b) \in U_{i-1} \times U_{j-1}} \mP_{ab}^n \right) \\
&\quad \quad \quad \quad + \sum_{(i, j) \not \in R'^2} \TV\left(\textsc{rk}_{(i-1)(j-1)}\left(\text{Bern}\left( Q \right) \right), \, \Motimes_{(a, b) \in U_{i-1} \times U_{j-1}} \mQ_{ab}^n \right) \\
&\quad \quad \le N^2 \cdot \Delta
\end{align*}
Now note that when $\tau = \tau'$ is chosen uniformly at random, the set $V$ is a uniformly at random chosen $k\ell$-subset of $[N\ell]$. Applying the conditioning property in Fact \ref{tvfacts} to conditioning on $R'$ and $\tau = \tau'$ now yields that
$$\TV\left( \mathcal{A}_3\left(\mathcal{M}_{N}\left( \text{Bern}(p), \text{Bern}\left(Q \right), k \right)\right), \, \mathcal{M}_{N\ell}( (\mP_{ij}, \mQ_{ij})_{1 \le i, j \le N\ell}, k\ell) \right) \le N^2 \cdot \Delta$$
Applying Lemma \ref{lem:tvacc} to the steps $\mathcal{A}_{\text{1-2}}$ and $\mathcal{A}_3$ with the sequence of distributions $\mathcal{P}_0 = \mG(n, k, p, q)$, $\mathcal{P}_{\text{1-2}} = \mathcal{M}_{N}\left( \text{Bern}(p), \text{Bern}\left(Q \right), k \right)$ and $\mathcal{P}_3 = \mathcal{M}_{N\ell}( (\mP_{ij}, \mQ_{ij})_{1 \le i, j \le N\ell}, k\ell)$ yields that
\begin{align*}
\TV\left( \mathcal{A}(\mG(n, k, p, q)), \, \mathcal{M}_{N\ell}( (\mP_{ij}, \mQ_{ij})_{1 \le i, j \le N\ell}, k\ell) \right) &\le N^2 \cdot \Delta + 4 \cdot \exp\left( - \frac{Q \epsilon^2 n^2}{32N} \right) \\
&\quad \quad + \sqrt{\frac{k^2(1 - Q)}{2QN}} + \sqrt{\frac{k^2Q}{2N(1 - Q)}}
\end{align*}

We now follow an analogous and simpler argument to analyze the case $G \sim \mG(n, q)$. Let $M^1 = \mathcal{A}_{\text{1-2}}(G)$ and note that $(G_1, G_2) \sim \mG(n, Q)^{\otimes 2}$ by Lemma \ref{lem:cloning}. Therefore the entries of $M^1$ are distributed as $M^1_{ij} \sim_{\text{i.i.d.}} \text{Bern}(Q)$ for all $i \neq j$ independently of $\text{diag}(M^1)$, which is an exchangeable distribution on $\{0, 1\}^N$ with support size $s_1 + \max\{s_2 - s_1, 0\}$ where $s_1 \sim \text{Bin}(n, p)$ and $s_2 \sim \text{Bin}(N, Q)$. Applying the tensorization and conditioning properties in Fact \ref{tvfacts} as in the previous case yields that
\begin{align*}
\TV\left( \mL(\mathcal{A}_{\text{1-2}}(\mG(n, q))), \text{Bern}\left(Q \right)^{\otimes N \times N} \right) &= \TV\left( \mL\left( \text{diag}(M^1) \right), \text{Bern}\left(Q \right)^{\otimes N} \right) \\
&= \TV\left( \mL( s_1 + \max\{s_2 - s_1, 0\} ), \text{Bin}\left(N, Q \right) \right) \\
&\le 4 \cdot \exp\left( - \frac{Q \epsilon^2 n^2}{32N} \right)
\end{align*}
by Lemma \ref{lem:plantingdiagonals}. Conditioning on $\tau = \tau'$ and applying the tensorization property in Fact \ref{tvfacts} yields
\begin{align*}
&\TV\left( \mL\left(\mathcal{A}_3\left( \text{Bern}\left(Q \right)^{\otimes N \times N} \right) \Big| \tau = \tau' \right), \, \mathcal{M}_{N\ell} \left((\mQ_{ij})_{1 \le i, j \le N\ell} \right) \right) \\
&\quad \quad \le \sum_{i, j = 1}^N \TV\left( \textsc{mrk}_{(i-1)(j-1)}\left( \text{Bern}\left( Q \right) \right), \, \Motimes_{(a, b) \in U_{i-1} \times U_{j-1}} \mQ_{ab}^n \right) \le N^2 \cdot \Delta
\end{align*}
Applying the conditioning property in Fact \ref{tvfacts} to conditioning on $\tau = \tau'$ now yields that
$$\TV\left( \mL\left(\mathcal{A}_3\left( \text{Bern}\left(Q \right)^{\otimes N \times N} \right)\right), \, \mathcal{M}_{N\ell} \left((\mQ_{ij})_{1 \le i, j \le N\ell} \right) \right) \le N^2 \cdot \Delta$$
Applying Lemma \ref{lem:tvacc} to $\mathcal{A}_{\text{1-2}}$ and $\mathcal{A}_3$ with distributions $\mathcal{P}_0 = \mG(n, q)$, $\mathcal{P}_{\text{1-2}} = \text{Bern}\left(Q \right)^{\otimes N \times N}$ and $\mathcal{P}_3 = \mathcal{M}_{N\ell} \left((\mQ_{ij})_{1 \le i, j \le N\ell} \right)$ yields that
$$\TV\left( \mathcal{A}(\mG(n, q)), \, \mathcal{M}_{N\ell} \left((\mQ_{ij})_{1 \le i, j \le N\ell} \right) \right) \le N^2 \cdot \Delta + 4 \cdot \exp\left( - \frac{Q \epsilon^2 n^2}{32N} \right)$$
which completes the proof of the theorem.

\section{Computational Barriers in Submatrix Detection}
\label{sec:compbarriers}

\subsection{Computational Lower Bounds from Our Average-Case Reduction}

The average-case reduction from planted dense subgraph in the previous section implies lower bounds for a more general heteroskedastic version of submatrix detection where the pairs of planted and noise distributions are allowed to vary from entry to entry, that we now formally define in the notation from the previous section. 

\begin{definition}[Heteroskedastic Symmetric Index Set Submatrix Detection]
Given computable pairs $(\mP_{ij}, \mQ_{ij})$ for $1 \le i, j \le n$ over a common measurable space $(X, \mathcal{B})$, define $\textsc{hssd}(n, k, (\mP_{ij}, \mQ_{ij})_{1 \le i, j \le n})$ to have observation $M \in X^{n \times n}$ and hypotheses
$$H_0 : M \sim \mathcal{M}_n\left( (\mQ_{ij})_{1 \le i, j \le n} \right) \quad \textnormal{and} \quad H_0 : M \sim \mathcal{M}_n\left( (\mP_{ij}, \mQ_{ij})_{1 \le i, j \le n}, k \right)$$
\end{definition}

The reduction $\textsc{To-Submatrix}$ from the previous section yields the following lower bounds for $\textsc{hssd}$ based on the $\textsc{pds}$ conjecture. We state the implied lower bounds when $k = \Omega(\sqrt{n})$ and $k = o(\sqrt{n})$ separately in the next two theorems.

\begin{theorem}[Heteroskedastic $\textsc{pds}$ Lower Bounds when $k = o(\sqrt{n})$] \label{thm:lowsparsehssd}
Let $k = o(\sqrt{n})$, let $0 < q < p \le 1$ be fixed constants and let $(\mP_{ij}, \mQ_{ij})$ be a computable pairs over $(X, \mathcal{B})$ for each $1 \le i, j \le n$ such that
$$\sup_{1 \le i, j \le n} \bP_{X \sim \mD_{ij}}\left[ \log \frac{d\mP_{ij}}{d\mQ_{ij}}(X) \not \in \left[ \log \left( \frac{1 - p}{1 - q} \right), \log \left( \frac{p}{q} \right) \right] \right] = o\left(n^{-2} \right)$$
under both the settings $\mD_{ij} = \mP_{ij}$ and $\mD_{ij} = \mQ$ for each $i, j \in [n]$. Then assuming the $\pr{pds}$ conjecture at densities $0 < q < p \le 1$, there is no randomized polynomial time algorithm solving $\textsc{hssd}(n, k, (\mP_{ij}, \mQ_{ij})_{1 \le i, j \le n})$ with asymptotic Type I$+$II error less than one.
\end{theorem}

\begin{proof}
Consider applying Lemma \ref{lem:3a} and Theorem \ref{thm:tosubmatrix} with $\ell = 1$ and starting planted dense subgraph instance with subgraph size $k$, $n$ vertices and densities $0 < q < p \le 1$. Excluding $n^2 \cdot \Delta$, all of the terms in both of the total variation upper bounds in Theorem \ref{thm:tosubmatrix} are $o(1)$ since $k^2 = o(n)$ and $Q < P$ are constants. It suffices to show that $n^2 \cdot \Delta = o(1)$. By the definition of $\Delta$, we have that $\Delta = \max_{i, j \in [n]} \Delta_{ij}$ since $\ell = 1$. Now consider the definition of $\Delta_{ij}$ in Lemma \ref{lem:mrk}. The condition in the theorem statement above implies that $\bP_{x \sim \mQ_n^*}[x \not \in S] = o(n^{-2})$ and $\bP_{x \sim \mP_n^*}[x \not \in S] = o(n^{-2})$ and in particular that they both are at most $\frac{1}{4}(p - q)$. By the same argument as in the proof of Lemma \ref{lem:mrkupperbound}, this is sufficient to imply that $\Delta_{ij} = o(n^{-2})$. Since this holds for each $i, j \in [n]$, we have that $\Delta = o(n^{-2})$, which proves the theorem.
\end{proof}

\begin{theorem}[Heteroskedastic $\textsc{pds}$ Lower Bounds when $k = \Omega(\sqrt{n})$] \label{thm:highsparsehssd}
Let $k = \Omega(\sqrt{n})$, let $0 < q < p \le 1$ be fixed constants and let $(\mP_{ij}, \mQ_{ij})$ be a computable pairs over $(X, \mathcal{B})$ for each $1 \le i, j \le n$. Suppose that there is some $m = \omega(k^2/n)$ with $m = o(n)$ such that
$$\sup_{S, T \subseteq [n] : |S| = |T| = m} \bP_{X_{S \times T} \sim \mD_{S \times T}}\left[ \sum_{i \in S} \sum_{j \in T} \log \frac{d\mP_{ij}}{d\mQ_{ij}}(X_{ij}) \not \in \left[ \log \left( \frac{1 - p}{1 - q} \right), \log \left( \frac{p}{q} \right) \right] \right] = o\left( \frac{m^2}{n^2} \right)$$
under both the settings $\mD_{S \times T} = \otimes_{i \in S} \otimes_{j \in T} \mP_{ij}$ and $\mD_{S \times T} = \otimes_{i \in S} \otimes_{j \in T} \mQ_{ij}$ for each $S, T$. Then assuming the $\pr{pds}$ conjecture at densities $0 < q < p \le 1$, there is no randomized polynomial time algorithm solving $\textsc{hssd}(n, k, (\mP_{ij}, \mQ_{ij})_{1 \le i, j \le n})$ with asymptotic Type I$+$II error less than one.
\end{theorem}

\begin{proof}
Consider applying Lemma \ref{lem:3a} and Theorem \ref{thm:tosubmatrix} with $\ell = m$ and starting planted dense subgraph instance with subgraph size $\lfloor k/m \rfloor$, $\lfloor n/m \rfloor$ vertices and densities $0 < q < p \le 1$. Since $m = \omega(k^2/n)$, it follows that $k/m = o(\sqrt{n/m})$ and thus it suffices to show that the total variation upper bounds in Theorem \ref{thm:tosubmatrix} are $o(1)$. As in the proof of the previous theorem, this reduces to showing that $(n/m)^2 \cdot \Delta = o(1)$. For each pair $S, T \subseteq [n]$ with $|S| = |T| = m$, consider the definition of $\Delta_{S, T}$ in Lemma \ref{lem:mrk}. We have that $\bP_{x \sim \mQ_n^*}[x \not \in S] = o(n^{-2})$ and $\bP_{x \sim \mP_n^*}[x \not \in S] = o(m^2/n^2)$ by the guarantees in the theorem statement. Since $m = o(n)$, these probabilities are at most $\frac{1}{4}(p - q)$ for large $n$. By the same reasoning as in the previous theorem, we have that $\Delta = o(m^2/n^2)$, proving the theorem.
\end{proof}

We remark that we ignored issues of divisibility in the previous theorem statement, reducing to an instance with $m \lfloor n/m \rfloor$ vertices and submatrix size $m \lfloor k/m \rfloor$ instead of exactly $n$ and $m$. This can be resolved by taking all of $n, k, m, \ell$ to be powers of two without affecting their sizes by more than a factor of $2$. Constructing a sequence of indices with these parameters is enough to rule out polynomial time algorithms given our forms of the $\pr{pc}$ and $\pr{pds}$ conjectures.

From this point forward, we will restrict our attention to the homoskedastic formulation of submatrix detection that we have so far focused on. However, we first remark that these general heteroskedastic lower bounds can be specialized to imply hardness for submatrix problem with dependences between entries induced by natural column-wise and row-wise mixtures. Let $\mP(\theta)$ be a family of distributions indexed by $\theta$ such that $(\mP(\theta), \mQ)$ is a computable pair for each $\theta \in \Theta$. Now consider the submatrix problem with $H_0 : M \sim \mQ^{\otimes n \times n}$ and $H_1$ distribution formed as follows:
\begin{itemize}
\item Select a subset $S \subseteq [n]$ with $|S| = k$ uniformly at random
\item Sample $\theta_i \sim_{\text{i.i.d.}} \mD$ for each $i \in [n]$
\item Sample $M_{ij} \sim \mP(\theta_i)$ for each $i, j \in S$ and $M_{ij} \sim \mQ$ otherwise independently
\end{itemize}
for some distribution $\mD$ on $\Theta$. For example, if $\mP(\theta) = \mN(\theta, 1)$, $\mQ = \mN(0, 1)$ and $\mD$ is normally distributed, this model has row-wise dependences resembling sparse PCA. Now suppose that an algorithm $\mathcal{A}$ solve this problem with asymptotic Type I$+$II error $\epsilon$, then there must be a deterministic choice of the $\theta_i$ such that $\mathcal{A}$ solves the problem with asymptotic Type I$+$II error $\epsilon$. When $\theta_i$ are deterministic, this problem is exactly $\textsc{hssd}(n, k, (\mP_{ij}, \mQ_{ij})_{1 \le i, j \le n})$ with $\mP_{ij} = \mP(\theta_i)$ and $\mQ_{ij} = \mQ$. 

We now combine the heteroskedastic lower bounds in Theorems \ref{thm:lowsparsehssd} and \ref{thm:highsparsehssd} with the $\textsc{mrk}$ upper bound on $\Delta$ given in Lemma \ref{lem:mrkupperbound} to yield clean statements of the implied computational lower bounds for $\pr{ssd}$ based on the $\textsc{pds}$ and $\textsc{pc}$ conjectures.

\begin{corollary}[$\textsc{pds}$ Lower Bounds for Submatrix Detection] \label{cor:pdslowerbounds}
Let $0 < q < p \le 1$ be fixed constants and $(\mP, \mQ)$ be a computable pair over $(X, \mathcal{B})$ such that either:
\begin{itemize}
\item $k = \Omega(\sqrt{n})$ and $\frac{k^4}{n^2} \cdot \SKL(\mP, \mQ) \to 0$ and the LLR between $(\mP, \mQ)$ satisfies the LDP
$$E_\mP\left(m \right) = \omega(m \log n) \quad \textnormal{and} \quad E_\mQ\left(-m \right) = \omega(m \log n)$$
for some positive $m$ satisfying $\max \left\{ \KL(\mQ \| \mP), \KL(\mP \| \mQ) \right\} \le m = o(n^2/k^4)$ and where the second inequality is only necessary if $p \neq 1$
\item $k = o(\sqrt{n})$, $\KL(\mQ \| \mP) < \log \left( \frac{1 - q}{1 - p} \right)$ and $\KL(\mP \| \mQ) < \log \left( \frac{p}{q} \right)$ and the LLR between $(\mP, \mQ)$ satisfies the LDP
$$E_\mP\left(\log \left( \frac{p}{q} \right) \right) \ge 2 \log n + \omega(1) \quad \textnormal{and} \quad E_\mQ\left( \log \left( \frac{1 - p}{1 - q} \right) \right) \ge 2 \log n + \omega(1)$$
where the second inequality is only necessary if $p \neq 1$
\end{itemize}
Then assuming the $\pr{pds}$ conjecture at densities $0 < q < p \le 1$, there is no randomized polynomial time algorithm solving $\pr{ssd}(n, k, \mP, \mQ)$ with asymptotic Type I$+$II error less than one.
\end{corollary}

\begin{proof}
We first consider the case where $k = \Omega(\sqrt{n})$. Consider the reduction in Theorem \ref{thm:highsparsehssd} with blow-up factor $\ell = \omega(k^2/n)$ where $\ell$ is chosen so that $\ell^{-2} = \omega(m)$ where $m$ is the positive constant in the statement of the corollary. Applying Lemma \ref{lem:mrkupperbound} to the $\textsc{mrk}$ with blow-up factor $\ell^2$ yields
$$\Delta \le \frac{3 \exp\left( - \ell^2 \cdot E_\mP\left(c_+ \ell^{-2}\right) \right) + 3 \exp\left( - \ell^2 \cdot E_\mQ\left( c_- \ell^{-2} \right) \right)}{p - q}$$
where $c_+ = \log \left( \frac{p}{q} \right)$ and $c_- = \log \left( \frac{1 - p}{1 - q} \right)$. Since $E_\mP(\lambda)$ is convex and minimized at $\KL(\mQ \| \mP)$, it follows that since $c_+ \ell^{-2} \ge m$ we have that
$$\frac{E_\mP\left(c_+ \ell^{-2}\right)}{c_+ \ell^{-2} - \KL(\mQ \| \mP)} \ge \frac{E_\mP\left(m \right)}{m - \KL(\mQ \| \mP)} = \omega(\log n)$$
Therefore $E_\mP\left(c_+ \ell^{-2}\right) = \omega(\ell^{-2} \log n)$. A symmetric argument shows that $E_\mQ\left( c_- \ell^{-2} \right) = \omega(\ell^{-2} \log n)$. Now it follows that $\Delta = o(n^{-2})$, which yields the first statement on applying Theorem \ref{thm:highsparsehssd}. Now consider the case where $k = o(\sqrt{n})$. Applying Lemma \ref{lem:mrkupperbound} to the $\textsc{mrk}$ with blow-up factor $1$ to the reduction in Theorem \ref{thm:lowsparsehssd} yields that
$$\Delta \le \frac{3 \exp\left( - E_\mP\left(c_+\right) \right) + 3 \exp\left( - E_\mQ\left( c_- \right) \right)}{p - q} = o(n^{-2})$$
by the given conditions. Combining this with Theorem \ref{thm:lowsparsehssd} completes the proof.
\end{proof}

Note that the constraints on $E_\mQ$ are no longer necessary if $p = 1$. The next corollary states our $\textsc{pc}$ lower bounds and is a restatement of the main theorem on computational lower bounds from Section \ref{sec:summary}.

\begin{corollary}[$\textsc{pc}$ Lower Bounds for Submatrix Detection] \label{cor:pclowerbounds}
Let $p \in (0, 1)$ be a fixed constant and $(\mP, \mQ)$ be a computable pair over $(X, \mathcal{B})$ such that either:
\begin{itemize}
\item $k = \Omega(\sqrt{n})$ and $\frac{k^4}{n^2} \cdot \KL(\mP \| \mQ) \to 0$ and the LLR between $(\mP, \mQ)$ satisfies the LDP
$$E_\mP\left(m \right) \ge \omega(m \log n)$$
for some positive $m$ with $\KL(\mP \| \mQ) \le m = o(n^2/k^4)$
\item $k = o(\sqrt{n})$ and $\KL(\mP \| \mQ) < \log p^{-1}$ and the LLR between $(\mP, \mQ)$ satisfies the LDP
$$E_\mP\left( \log p^{-1} \right) \ge 2\log n + \omega(1)$$
\end{itemize}
Then assuming the $\pr{pc}$ conjecture at density $p$, there is no randomized polynomial time algorithm solving $\pr{ssd}(n, k, \mP, \mQ)$ with asymptotic Type I$+$II error less than one.
\end{corollary}

We now give another corollary of Theorems \ref{thm:lowsparsehssd} and \ref{thm:highsparsehssd} yielding a slight variation of these lower bounds for submatrix detection based on the $\pr{pc}$ conjecture. This corollary implies that the lower bounds stated in Section \ref{sec:completephasediagram} hold for all $(\mP, \mQ)$ in the universality class $\pr{uc-a}$. Note that unlike the previous corollary which yielded clean lower bounds given the $\pr{pc}$ conjecture for any fixed $p$, we deduce the desired lower bounds up to a factor of $n^{\epsilon(p)}$ where $\epsilon(p)$ tends to zero with $p$. The proof of this corollary is very similar to that of Corollary \ref{cor:pdslowerbounds}, making crucial use of the convexity of $E_\mP$ and the definition of $\pr{uc-a}$.

\begin{corollary}[Computational Lower Bounds for $\pr{uc-a}$]
Suppose that $(\mP, \mQ)$ is a computable pair in $\pr{uc-a}$. Fix any $\epsilon \in (0, 1)$ and suppose that either:
\begin{itemize}
\item $k = \Omega(\sqrt{n})$ and $\KL(\mP \| \mQ) = o\left( \frac{n^{2- \epsilon}}{k^4} \right)$ or
\item $k = o(\sqrt{n})$ and $\KL(\mP \| \mQ) = o\left(n^{-\epsilon}\right)$
\end{itemize}
Then there is a sufficiently small $p = p(\epsilon) > 0$ such that assuming the $\pr{pc}$ conjecture at density $p$, there is no randomized polynomial time algorithm solving $\pr{ssd}(n, k, \mP, \mQ)$ with asymptotic Type I$+$II error less than one.
\end{corollary}

\begin{proof}
We first consider the case where $k = \Omega(\sqrt{n})$. Consider the reduction in Theorem \ref{thm:highsparsehssd} with blow-up factor $\ell = \Theta(k^2/n^{1 - \epsilon/4})$. Observe that by the given assumption, we have that $n^{\epsilon/2} \cdot \KL(\mP \| \mQ) = o(\ell^{-2})$. Applying Lemma \ref{lem:mrkupperbound} to the $\textsc{mrk}$ with blow-up factor $\ell^2$ yields
$$\Delta \le \frac{3 \exp\left( - \ell^2 \cdot E_\mP\left(\ell^{-2} \log p^{-1}\right) \right)}{1 - p}$$
Since $E_\mP(\lambda)$ is convex and minimized at $\KL(\mQ \| \mP)$, it follows that since $\ell^{-2} \log p^{-1} \ge n^{\epsilon/2} \cdot \KL(\mP \| \mQ)$ we have that
$$\frac{E_\mP\left(\ell^{-2} \log p^{-1}\right)}{\ell^{-2} \log p^{-1} - \KL(\mQ \| \mP)} \ge \frac{E_\mP\left(n^{\epsilon/2} \cdot \KL(\mP \| \mQ) \right)}{n^{\epsilon/2} \cdot \KL(\mP \| \mQ) - \KL(\mQ \| \mP)} \ge c \cdot \log n$$
for some constant $c > 0$ since $(\mP, \mQ)$ is in $\pr{uc-a}$. Using the fact that $\KL(\mQ \| \mP) \le \frac{1}{2} ell^{-2} \log p^{-1}$ for sufficiently large $n$, we have that
$$\Delta \le \exp\left( \frac{c}{2} \log p^{-1} \cdot \log n \right) = o(n^{-2})$$
if $p$ is taken to be sufficiently small. The first statement follows on applying Theorem \ref{thm:highsparsehssd}. Now consider the case where $k = o(\sqrt{n})$. Applying Lemma \ref{lem:mrkupperbound} to the $\textsc{mrk}$ with blow-up factor $1$ to the reduction in Theorem \ref{thm:lowsparsehssd} yields that
$$\Delta \le \frac{3 \exp\left( - E_\mP\left(\log p^{-1} \right) \right)}{1 - q}$$
by the given conditions. We now apply a similar convexity step, noting that since $n^{-\epsilon} \cdot \KL(\mP \| \mQ) \le \log p^{-1}$ for sufficiently large $n$, we have that
$$\frac{E_\mP\left(\log p^{-1}\right)}{\log p^{-1} - \KL(\mQ \| \mP)} \ge \frac{E_\mP\left(n^{\epsilon} \cdot \KL(\mP \| \mQ) \right)}{n^{\epsilon} \cdot \KL(\mP \| \mQ) - \KL(\mQ \| \mP)} \ge c' \cdot \log n$$
for some $c' > 0$. Now it similarly follows that $\Delta \le \exp\left( \frac{c'}{2} \log p^{-1} \cdot \log n \right) = o(n^{-2})$ if $p$ is taken to be sufficiently small. Combining this with Theorem \ref{thm:lowsparsehssd} completes the proof.
\end{proof}

\subsection{Polynomial Time Test Statistics for Submatrix Detection}
\label{sec:polytests}

In this section, we show algorithmic upper bounds for submatrix detection using two simple test statistics that can be computed in polynomial time. Given a computable pair of distributions $(\mP, \mQ)$ over the measurable space $(X, \mathcal{B})$ and a matrix $M \in X^{n \times n}$, define
\begin{align*}
T_{\text{sum}}(M) &= \frac{1}{n^2} \sum_{i, j = 1}^n \log \frac{d\mP}{d\mQ}(M_{ij}) \\
T_{\text{max}}(M) &= \max_{1 \le i, j \le n} \log \frac{d\mP}{d\mQ}(M_{ij})
\end{align*}
Note that both $T_{\text{sum}}$ and $T_{\text{max}}$ can be computed in $O(n^2 \cdot \mathcal{T})$ time where the Radon-Nikodym derivative $\frac{d\mP}{d\mQ}(M_{ij})$ can be evaluated in $O(\mathcal{T})$ time. Given that $(\mP, \mQ)$ is a computable pair, it follows that $T_{\text{sum}}$ and $T_{\text{max}}$ can be computed in polynomial time. We now show that thresholding these statistics solves the asymmetric detection problem $\textsc{asd}$ given sufficient LDPs for the LLR under each of $\mQ$ and $\mP$. We begin with the sum test $T_{\text{sum}}$.

\begin{proposition}[Sum Test] \label{prop:sumtest}
Let $M$ be an instance of $\pr{asd}(n, k, \mP, \mQ)$ and let
$$\tau_{\textnormal{sum}} = -\KL(\mQ \| \mP) + \frac{k^2}{2n^2} \cdot \SKL(\mP, \mQ)$$
Suppose that $k \ll n$ and
\begin{align*}
E_\mP\left( \frac{1}{2} \cdot \KL(\mP \| \mQ) \right) &= \omega(k^{-2}) \\
E_\mQ\left( -\frac{2n^2 - k^2}{2n^2 - 2k^2} \cdot \KL(\mQ \| \mP) \right) &= \omega(n^{-2}) \\
E_\mQ\left( -\frac{2n^2 - k^2}{2n^2} \cdot \KL(\mQ \| \mP) \right) &= \omega(n^{-2})
\end{align*}
Then $\bP_{H_0}\left[ T_{\textnormal{sum}}(M) \ge \tau_{\textnormal{sum}} \right] \to 0$ and $\bP_{H_1}\left[ T_{\textnormal{sum}}(M) < \tau_{\textnormal{sum}} \right] \to 0$ as $n \to \infty$.
\end{proposition}

\begin{proof}
Let $\tau' = -\frac{2n - k^2}{2n} \cdot \KL(\mQ \| \mP)$ and note that $\tau' \le \tau_{\text{sum}}$. Under $H_0$, by a Chernoff bound we have that if $\lambda \ge 0$ then
\begin{align*}
\bP_{H_0}\left[ T_{\textnormal{sum}}(M) \ge \tau_{\textnormal{sum}} \right] &\le \bP_{H_0}\left[ T_{\textnormal{sum}}(M) \ge \tau' \right] \\
&= \bP_{H_0}\left[ \exp\left(n^2 \lambda \cdot T_{\textnormal{sum}}(M) \right) \ge \exp\left( n^2\lambda \cdot \tau' \right) \right] \\
&\le \exp\left( n^2 \cdot \psi_{\mQ}\left(\lambda\right) -n^2 \lambda \cdot \tau' \right)
\end{align*}
Since $\tau' \ge -\KL(\mQ \| \mP)$, we may take $\lambda \ge 0$ so that $\lambda \cdot \tau' - \psi_{\mQ}(\lambda)$ is arbitrarily close to $E_\mQ(\tau')$. Therefore we have that
$$\bP_{H_0}\left[ T_{\textnormal{sum}}(M) \ge \tau_{\textnormal{sum}} \right] \le \exp\left( - n^2 \cdot E_\mQ(\tau') \right) = o(1)$$
Let $S', T' \subseteq [n]$ be the latent row and column indices of the planted part of $M$ under $H_1$. Note that $\tau_{\text{sum}} = (1 - \frac{k^2}{n^2}) \cdot \tau_1 + \frac{k^2}{n^2} \cdot \tau_2$ where $\tau_1 = -\frac{2n^2 - k^2}{2n^2 - 2k^2} \cdot \KL(\mQ \| \mP)$ and $\tau_2 = \frac{1}{2} \cdot \KL(\mP \| \mQ)$. Thus under $H_1$, by a Chernoff and union bound we have that if $\lambda_1, \lambda_2 \le 0$
\begin{align*}
\bP_{H_1}\left[ T_{\textnormal{sum}}(M) < \tau_{\textnormal{sum}} \right] &\le \bP_{H_1}\left[ \sum_{(i, j) \not \in S' \times T'} \log \frac{d\mP}{d\mQ}(M_{ij}) < \left(1 - \frac{k^2}{n^2}\right) \cdot \tau_1 \right] \\
&\quad \quad + \bP_{H_1}\left[ \sum_{(i, j) \in S' \times T'} \log \frac{d\mP}{d\mQ}(M_{ij}) < \frac{k^2}{n^2} \cdot \tau_2 \right] \\
&\le \exp\left( -(n^2 - k^2) \lambda_1 \cdot \tau_1 \right) \cdot \bE_{H_1} \left[ \exp\left( \lambda_1 \cdot \sum_{(i, j) \not \in S' \times T'} \log \frac{d\mP}{d\mQ}(M_{ij}) \right) \right] \\
&\quad \quad + \exp\left( -k^2 \lambda_2 \cdot \tau_2 \right) \cdot \bE_{H_1} \left[ \exp\left( \lambda_2 \cdot \sum_{(i, j) \in S' \times T'} \log \frac{d\mP}{d\mQ}(M_{ij}) \right) \right] \\
&= \exp\left( (n^2 - k^2) \cdot \psi_{\mQ}(\lambda_1) - (n^2 - k^2) \cdot \lambda_1 \cdot \tau_1 \right) +  \exp\left( k^2 \cdot \psi_{\mP}(\lambda_2) - k^2 \cdot \lambda_2 \cdot \tau_2 \right)
\end{align*}
Since $\tau_1 \le - \KL(\mQ \| \mP)$ and $\tau_2 \le \KL(\mP \| \mQ)$, we may take $\lambda_1 \le 0$ and $\lambda_2 \le 0$ so that $\lambda_1 \cdot \tau_1 - \psi_{\mQ}(\lambda_1)$ is arbitrarily close to $E_\mQ(\tau_1)$ and $\lambda_2 \cdot \tau_2 - \psi_{\mP}(\lambda_2)$ is arbitrarily close to $E_\mP(\tau_2)$. This yields that
$$\bP_{H_1}\left[ T_{\textnormal{sum}}(M) < \tau_{\textnormal{sum}} \right] \le \exp\left( - (n^2 - k^2) \cdot E_\mQ(\tau_1) \right) + \exp\left( - k^2 \cdot E_\mP(\tau_2) \right) = o(1)$$
which completes the proof of the proposition.
\end{proof}

Given an LDP for the LLR under $\mQ$ and $\mP$, it also holds that $T_{\text{sum}}$ solves the asymmetric detection problem.

\begin{proposition}[Max Test] \label{prop:maxtest}
Let $M$ be an instance of $\pr{asd}(n, k, \mP, \mQ)$ and suppose there is a $\tau_{\textnormal{max}} \in \left(- \KL(\mQ \| \mP), \KL(\mP \| \mQ) \right)$ with 
$$E_\mQ(\tau_{\textnormal{max}}) \ge 2 \log n + \omega(1) \quad \textnormal{and} \quad E_\mP(\tau_{\textnormal{max}}) = \omega(1)$$
then $\bP_{H_0}\left[ T_{\textnormal{max}}(M) \ge \tau_{\textnormal{max}} \right] \to 0$ and $\bP_{H_1}\left[ T_{\textnormal{max}}(M) < \tau_{\textnormal{max}} \right] \to 0$ as $n \to \infty$.
\end{proposition}

\begin{proof}
This follows from the same argument used to analyze the search test $T_{\text{search}}$ in the proof of Proposition \ref{prop:searchtest} applied with $k = 1$.
\end{proof}

We now proceed to show that our algorithmic upper bounds hold for all computable pairs $(\mP, \mQ)$ in $\pr{uc-b}$. To do this, we will establish the following simple consequences for $(\mP, \mQ)$ in $\pr{uc-b}$. As mentioned in Section \ref{sec:summary}, the class $\pr{uc-b}$ is introduced in \cite{hajek2017information}. The third property below is derived in Section 3 of \cite{hajek2017information}.

\begin{lemma}[Properties of $\textsc{uc-b}$] \label{lem:UCB}
Suppose that $(\mP, \mQ)$ is in $\pr{uc-a}$ with constant $C \ge 1$. Then
\begin{enumerate}
\item It holds for all $\tau \in [-2C \cdot \KL(\mP \| \mQ), 0]$ that
$$E_\mP\left(\KL(\mP \| \mQ) + \tau \right) \ge \frac{\tau^2}{4C \cdot \KL(\mP \| \mQ)}$$
\item It holds for all $\tau \in [-2C \cdot \KL(\mQ \| \mP), 2C \cdot \KL(\mQ \| \mP)]$ that
$$E_\mQ\left(-\KL(\mQ \| \mP) + \tau \right) \ge \frac{\tau^2}{4C \cdot \KL(\mQ \| \mP)}$$
\item It holds that $\KL(\mP \| \mQ) = \Theta(\KL(\mQ \| \mP))$.
\end{enumerate}
\end{lemma}

\begin{proof}
The fact that $(\mP, \mQ)$ is in $\pr{uc-b}$ implies that
\begin{align*}
E_\mP\left(\KL(\mP \| \mQ) + \tau \right) &= \sup_{\lambda \in \mathbb{R}} \left\{ \left(\KL(\mP \| \mQ) + \tau \right) \cdot \lambda - \psi_\mP(\lambda) \right\} \\
&\ge \sup_{\lambda \in [-1, 0]} \left\{ \tau \cdot \lambda - C \cdot \KL(\mP \| \mQ) \cdot \lambda^2 \right\}
\end{align*}
Now set $\lambda = \frac{\tau}{2C \cdot \KL(\mP \| \mQ)} \in [-1, 0]$ and note that this implies that
$$E_\mP\left(\KL(\mP \| \mQ) + \tau \right) \ge \frac{\tau^2}{4C \cdot \KL(\mP \| \mQ)}$$
Similarly, we have that
$$E_\mQ\left(-\KL(\mQ \| \mP) + \tau \right) \ge \sup_{\lambda \in [-1, 1]} \left\{ \tau \cdot \lambda - C \cdot \KL(\mQ \| \mP) \cdot \lambda^2 \right\} \ge \frac{\tau^2}{4C \cdot \KL(\mQ \| \mP)}$$
on setting $\lambda = \frac{\tau}{2C \cdot \KL(\mQ \| \mP)} \in [-1, 1]$. Property 3 follows from Lemma 2 in \cite{hajek2017information}, which shows that
$$\min\left\{ \KL(\mQ \| \mP), \KL(\mP \| \mQ) \right\} \ge \frac{1}{C} \cdot \max\left\{ \KL(\mQ \| \mP), \KL(\mP \| \mQ) \right\}$$
This implies that $\KL(\mP \| \mQ) = \Theta(\KL(\mQ \| \mP))$.
\end{proof}

We now combine these properties with Propositions \ref{prop:sumtest} and \ref{prop:maxtest} to show algorithmic achievability of the computational barriers shown above for $\textsc{ssd}$ when $(\mP, \mQ)$ is in $\pr{uc-b}$.

\begin{corollary}[Algorithmic Upper Bounds for $\pr{uc-b}$]
Suppose that $(\mP, \mQ)$ is a computable pair in $\pr{uc-b}$. Then it follows that:
\begin{itemize}
\item If $k = o(n)$, $k = \Omega(\sqrt{n})$ and $\SKL(\mP, \mQ) = \omega\left( \frac{n^2}{k^4} \right)$, then $T_{\textnormal{sum}}$ solves $\pr{ssd}(n, k, \mP, \mQ)$.
\item If $k = o(\sqrt{n})$ and $\SKL(\mP, \mQ) \ge c \cdot \log n$ for some sufficiently large constant $c > 0$, then $T_{\textnormal{max}}$ with $\tau_{\textnormal{max}} = 0$ solves $\pr{ssd}(n, k, \mP, \mQ)$.
\end{itemize}
\end{corollary}

\begin{proof}
We begin with the first statement. By Proposition \ref{prop:sumtest}, it suffices to verify the lower bounds on $E_\mP$ and $E_\mQ$ in the statement of the proposition. To do this, we apply properties (1) and (2) in Lemma \ref{lem:UCB}. Let $C \ge 1$ be the constant for which $(\mP, \mQ)$ is in $\pr{uc-b}$ and observe that
\begin{align*}
E_\mP\left( \frac{1}{2} \cdot \KL(\mP \| \mQ) \right) &\ge \frac{\left( - \frac{1}{2} \cdot \KL(\mP \| \mQ) \right)^2}{4C \cdot \KL(\mP \| \mQ)} = \frac{1}{16C} \cdot \KL(\mP \| \mQ) = \omega(k^{-2}) \\
E_\mQ\left( -\frac{2n^2 - k^2}{2n^2 - 2k^2} \cdot \KL(\mQ \| \mP) \right) &\ge \frac{\left( \frac{k^2}{2n^2 - 2k^2} \cdot \KL(\mQ \| \mP) \right)^2}{4C \cdot \KL(\mQ \| \mP)} = \frac{k^4}{16C(n^2 - k^2)^2} \cdot \KL(\mQ \| \mP) = \omega(n^{-2}) \\
E_\mQ\left( -\frac{2n^2 - k^2}{2n^2} \cdot \KL(\mQ \| \mP) \right) &\ge \frac{\left( \frac{k^2}{2n^2} \cdot \KL(\mQ \| \mP) \right)^2}{4C \cdot \KL(\mQ \| \mP)} = \frac{k^4}{16C \cdot n^4} \cdot \KL(\mQ \| \mP) = \omega(n^{-2})
\end{align*}
since $k = o(n)$, $\KL(\mP \| \mQ) = \Theta(\KL(\mQ \| \mP))$ and $\SKL(\mP, \mQ) = \omega\left( \frac{n^2}{k^4} \right)$. We now verify the second statement for $\tau_{\textnormal{max}} = 0$. It suffices to verify the two lower bounds on $E_\mP$ and $E_\mQ$ in the statement of Proposition \ref{prop:maxtest}. Note that
\begin{align*}
E_\mP(0) &\ge \frac{\left( - \KL(\mP \| \mQ) \right)^2}{4C \cdot \KL(\mP \| \mQ)} = \frac{1}{4C} \cdot \KL(\mP \| \mQ) \ge 3 \log n \\
E_\mQ(0) &\ge \frac{\left( \KL(\mQ \| \mP) \right)^2}{4C \cdot \KL(\mQ \| \mP)} = \frac{1}{4C} \cdot \KL(\mQ \| \mP) \ge 3 \log n
\end{align*}
for sufficiently large $c > 0$. This completes the proof of the corollary.
\end{proof}

\section{Statistical Limit of Submatrix Detection}
\label{sec:statbarriers}

In this section, we show information-theoretic lower bounds for our universal formulations of submatrix detection and provide a test statistic showing that this boundary is achievable.

\subsection{Information-Theoretic Lower Bound for Submatrix Detection}
\label{sec:itlowerbounds}

Assuming that $\mP$ and $\mQ$ have finite $\chi^2$ divergence, we can obtain the following information-theoretic lower bound for $\pr{ssd}$ with the distribution pair $(\mP, \mQ)$. The proof uses a similar $\chi^2$ divergence computation as in Lemma \ref{lem:hidingentries} and the information-theoretic lower bounds for planted dense subgraph shown in \cite{hajek2015computational}.

\begin{theorem} \label{thm:itlowerbounds}
Suppose that $\mP$ and $\mQ$ are probability distributions on a measurable space $(X,\mathcal{B})$ where $\mP$ is absolutely continuous with respect to $\mQ$. If $\chi^2( \mP\| \mQ)$ is finite and satisfies that
$$\chi^2( \mP \| \mQ) < \frac{1}{16e} \left( \frac{1}{n} \log \left( \frac{en}{k} \right) \wedge \frac{n^2}{k^4} \right)$$
then there is a function $\tau : \mathbb{R}_{> 0} \to \mathbb{R}_{> 0}$ such that $\lim_{t \to 0^+} \tau(t) = 0$ and
$$\TV\left( \mathcal{M}_n(\mP, \mQ, k), \, \mQ^{\otimes n \times n} \right) \le \tau \left( \frac{\chi^2( \mP \| \mQ)}{\frac{1}{n} \log \left( \frac{en}{k} \right) \wedge \frac{n^2}{k^4}} \right)$$
\end{theorem}

To prove this, we will need the following lemma of \cite{hajek2015computational} bounding the moment generating function of a hypergeometric random variable squared.

\begin{lemma}[Lemma 6 in \cite{hajek2015computational}] \label{lem:mgfhyp}
There exists a function $\tau_1 : \mathbb{R}_{> 0} \to \mathbb{R}_{> 0}$ satisfying that $\lim_{t \to 0^+} \tau_1(t) = 1$ such that for any $k \le n$ the following holds: if $H \sim \textnormal{Hypergeometric}(n, k, k)$ and $\lambda = \kappa \left( \frac{1}{k} \log \left( \frac{en}{k} \right) \wedge \frac{n^2}{k^4} \right)$ where $0 < \kappa < \frac{1}{16e}$ then
$$\bE\left[ \exp\left( \lambda H^2 \right) \right] \le \tau_1(\kappa)$$
\end{lemma}

Using this upper bound, we now can prove the information-theoretic lower bounds for submatrix detection in Theorem \ref{thm:itlowerbounds}.

\begin{proof}[Proof of Theorem \ref{thm:itlowerbounds}]
Let $f : X \to [0, \infty)$ be the Radon-Nikodym derivative $f = \frac{d\mathcal{P}}{d\mathcal{Q}}$. Observe that $\mathcal{M}_n(\mP, \mQ, k)$ can be written as the mixture
$$\mathcal{M}_n(\mP, \mQ, k) = \binom{n}{k}^{-1} \sum_{S \subseteq [n] : |S| = k} \mathcal{R}_S$$
where $\mathcal{R}_S$ is the distribution of $n \times n$ matrices $M \in X^{n \times n}$ with independent entries such that $M_{ij} \sim \mP$ if $i, j \in S$ and $M_{ij} \sim \mQ$ otherwise. Note that $\mathcal{R}_S$ is therefore absolutely continuous with respect to $\mQ^{\otimes n \times n}$ with Radon-Nikodym derivative $\frac{d \mathcal{R}_S}{d\mQ^{\otimes n \times n}}(M) = \prod_{i, j \in S} f(M_{ij})$ for each $M \in X^{n \times n}$. It follows that $\mathcal{M}_n(\mP, \mQ, k)$ is also absolutely continuous with respect to $\mQ^{\otimes n \times n}$ with Radon-Nikodym derivative
$$\frac{d\mathcal{M}_n(\mathcal{P}, \mathcal{Q}, k)}{d\mathcal{Q}^{\otimes n \times n}} (M) =  \binom{n}{k}^{-1} \sum_{S \subseteq [n] : |S| = k} \frac{d \mathcal{R}_S}{d\mQ^{\otimes n \times n}}(M) = \bE_{S \sim \mathcal{U}_{k, n}} \left[ \prod_{i, j \in S} f(M_{ij}) \right]$$
for each $x \in X^{n \times n}$ where $\mathcal{U}_{k, n}$ is the uniform distribution on $k$-subsets of $[n]$. By Fubini's theorem,
\begin{align*}
&\chi^2\left( \mathcal{M}_n(\mathcal{P}, \mathcal{Q}, k) \, \| \, \mathcal{Q}^{\otimes n \times n} \right) + 1 \\
&\quad \quad = \bE_{M \sim \mathcal{Q}^{\otimes n \times n}} \left[ \left( \frac{d\mathcal{M}_n(\mathcal{P}, \mathcal{Q}, k)}{d\mathcal{Q}^{\otimes n \times n}} (M) \right)^2 \right] \\
&\quad \quad = \bE_{M \sim \mathcal{Q}^{\otimes n \times n}} \left[ \bE_{S \sim \mathcal{U}_{k, n}} \left[ \prod_{i, j \in S} f(M_{ij}) \right] \cdot \bE_{T \sim \mathcal{U}_{k, n}} \left[ \prod_{i, j \in S} f(M_{ij}) \right] \right] \\
&\quad \quad = \bE_{S, T \sim \mathcal{U}_{k, n}} \left[ \bE_{M \sim \mathcal{Q}^{\otimes n \times n}} \left[ \left( \prod_{i, j \in S} f(M_{ij}) \right) \left( \prod_{i, j \in T} f(M_{ij}) \right) \right] \right] \\
&\quad \quad = \bE_{S, T \sim \mathcal{U}_{k, n}} \left[ \prod_{i, j \in S \cap T} \bE_{M_{ij} \sim \mathcal{Q}}\left[f(M_{ij})^2\right] \prod_{(i, j) \in S^2 \cup T^2 - (S \cap T)^2} \bE_{M_{ij} \sim \mathcal{Q}}\left[f(M_{ij}) \right] \right] \\
&\quad \quad = \bE_{S, T \sim \mathcal{U}_{k, n}} \left[ (1 + \chi^2(\mathcal{P} \| \mathcal{Q}))^{|S \cap T|^2} \right]
\end{align*}
where the last equality holds since $\bE_{M_{ij} \sim \mathcal{Q}}[f(M_{ij})] = 1$ and $1 + \chi^2(\mathcal{P} \| \mathcal{Q}) = \bE_{M_{ij} \sim \mathcal{Q}}[f(M_{ij})^2]$. Note that $H = |S \cap T| \sim \text{Hypergeometric}(n, k, k)$. Let $\tau_1$ be the function in Lemma \ref{lem:mgfhyp}. The given bounds on $\chi^2(\mP \| \mQ)$ imply that we can apply Lemma \ref{lem:mgfhyp} with $\lambda = \chi^2(\mP \| \mQ)$. Combining this with Cauchy-Schwarz and the fact that $1 + \chi^2(\mathcal{P} \| \mathcal{Q}) \le \exp\left( \chi^2(\mathcal{P} \| \mathcal{Q}) \right)$ yields that
\begin{align*}
2 \cdot \TV\left( \mathcal{M}_n(\mathcal{P}, \mathcal{Q}, k), \, \mathcal{Q}^{\otimes n \times n} \right)^2 &\le \chi^2\left( \mathcal{M}_n(\mathcal{P}, \mathcal{Q}, k) \, \| \, \mathcal{Q}^{\otimes n \times n} \right) \\
&= \bE_{S, T \sim \mathcal{U}_{k, n}} \left[ (1 + \chi^2(\mathcal{P} \| \mathcal{Q}))^{|S \cap T|^2} \right] - 1 \\
&\le \bE\left[ \exp\left( H^2 \cdot \chi^2(\mathcal{P} \| \mathcal{Q}) \right) \right] - 1 \\
&\le \tau_1 \left( \frac{\chi^2( \mP \| \mQ)}{\frac{1}{n} \log \left( \frac{en}{k} \right) \wedge \frac{n^2}{k^4}} \right) - 1
\end{align*}
Setting $\tau = \sqrt{\frac{1}{2}(\tau_1 - 1)}$ which satisfies $\lim_{t \to 0^+} \tau(t) = 0$ completes the proof of the theorem.
\end{proof}

Now using the fact that the minimum Type I$+$II error of a hypothesis testing problem between $\mL_0$ and $\mL_1$ is $1 - \TV(\mL_0, \mL_1)$, we arrive at the following corollary providing a regime in which submatrix detection is statistically impossible.

\begin{corollary} \label{cor:itlowerbounds}
Suppose that $\mP$ and $\mQ$ are probability distributions on a measurable space $(X,\mathcal{B})$ where $\mP$ is absolutely continuous with respect to $\mQ$. If $\chi^2( \mP\| \mQ)$ is finite and satisfies that
$$\chi^2( \mP \| \mQ) = o\left( \frac{1}{k} \log \left( \frac{n}{k} \right) \wedge \frac{n^2}{k^4} \right)$$
then any test $\phi : X^{n \times n} \to \{0, 1\}$ has an asymptotic Type I$+$II error of at least one on $\textsc{ssd}(n, k, \mP, \mQ)$.
\end{corollary}

Note that this corollary implies that if $(\mP, \mQ)$ is in $\pr{uc-c}$ and satisfies that $\chi^2(\mP \| \mQ) = O(\SKL(\mP, \mQ))$, then $\pr{ssd}(n, k, \mP, \mQ)$ is information-theoretically impossible if
$$\SKL(\mP, \mQ) = o\left( \frac{1}{k} \log \left( \frac{n}{k} \right) \wedge \frac{n^2}{k^4} \right)$$
which matches the bounds in Section \ref{sec:completephasediagram}.

\subsection{Search Test Statistic}
\label{sec:searchtest}

In this section, we give a simple search test statistic showing statistical achievability. Given a computable pair of distributions $(\mP, \mQ)$ over the measurable space $(X, \mathcal{B})$ and a matrix $M \in X^{n \times n}$, define
$$T_{\text{search}}(M) = \max_{S, T \subseteq [n] : |S| = |T| = k} \left( \frac{1}{k^2} \sum_{i \in S} \sum_{j \in T} \log \frac{d\mP}{d\mQ}(M_{ij}) \right)$$
Note that $T_{\text{search}}$ can be computed in $O(n^{2k} \cdot \mathcal{T})$ time where the Radon-Nikodym derivative $\frac{d\mP}{d\mQ}(M_{ij})$ can be evaluated in $O(\mathcal{T})$ time. We now show that thresholding this statistics solves the asymmetric detection problem $\textsc{asd}$ given sufficient LDPs for the LLR under each of $\mQ$ and $\mP$. We begin with the sum test $T_{\text{sum}}$.

\begin{proposition}[Search Test] \label{prop:searchtest}
Let $M$ be an instance of $\pr{asd}(n, k, \mP, \mQ)$ and suppose there is a $\tau_{\textnormal{search}} \in \left(- \KL(\mQ \| \mP), \KL(\mP \| \mQ) \right)$ with
$$E_\mQ(\tau_{\textnormal{search}}) \ge \frac{2}{k} \log \left( \frac{n}{k} \right) + \omega(k^{-2}) \quad \textnormal{and} \quad E_\mP(\tau_{\textnormal{search}}) = \omega(k^{-2})$$
then $\bP_{H_0}\left[ T_{\textnormal{search}}(M) \ge \tau_{\textnormal{search}} \right] \to 0$ and $\bP_{H_1}\left[ T_{\textnormal{search}}(M) < \tau_{\textnormal{search}} \right] \to 0$ as $n \to \infty$.
\end{proposition}

\begin{proof}
By a union bound and Chernoff bound, we have that for if $\lambda \ge 0$ then
\begin{align*}
&\bP_{H_0}\left[ T_{\textnormal{search}}(M) \ge \tau_{\textnormal{search}} \right] \\
&\quad \quad = \sum_{S, T \subseteq [n] : |S| = |T| = k} \bP_{H_0}\left[ \sum_{i \in S} \sum_{j \in T} \log \frac{d\mP}{d\mQ}(M_{ij}) \ge k^2 \cdot \tau_{\textnormal{search}} \right] \\
&\quad \quad \le \binom{n}{k}^2 \cdot \bP_{H_0}\left[ \exp\left( \lambda \cdot \sum_{i, j = 1}^k \log \frac{d\mP}{d\mQ}(M_{ij}) \right) \ge \exp\left( \lambda \cdot k^2 \cdot \tau_{\textnormal{search}} \right) \right] \\
&\quad \quad \le \exp\left( 2 \log \binom{n}{k} + k^2 \cdot \psi_{\mQ}\left(\lambda\right) -k^2 \lambda \cdot \tau_{\textnormal{search}} \right)
\end{align*}
Since $\tau_{\textnormal{search}} \in \left(- \KL(\mQ \| \mP), \KL(\mP \| \mQ) \right)$, we may take $\lambda \ge 0$ so that $\lambda \cdot \tau_{\text{search}} - \psi_{\mQ}(\lambda)$ is arbitrarily close to $E_\mQ(\tau_{\textnormal{search}})$. This implies that
$$\bP_{H_0}\left[ T_{\textnormal{search}}(M) \ge \tau_{\textnormal{search}} \right] \le \exp\left( 2k \cdot \log \left( \frac{n}{k} \right) - k^2 \cdot E_\mQ(\tau_{\textnormal{search}}) \right) = o(1)$$
since $\binom{n}{k} \le \left( \frac{n}{k} \right)^k$. Let $S', T' \subseteq [n]$ be the latent row and column indices of the planted part of $M$ under $H_1$. Now it follows that for $\lambda \le 0$ we have that
\begin{align*}
\bP_{H_1}\left[ T_{\textnormal{search}}(M) < \tau_{\textnormal{search}} \right] &= \bP_{H_1}\left[ \sum_{i \in S'} \sum_{j \in T'} \log \frac{d\mP}{d\mQ}(M_{ij}) < k^2 \cdot \tau_{\textnormal{search}} \right] \\
&\le \bP_{H_1}\left[ \exp\left( \lambda \cdot \sum_{i \in S'} \sum_{j \in T'} \log \frac{d\mP}{d\mQ}(M_{ij}) \right) > \exp\left(  \lambda \cdot k^2 \cdot \tau_{\textnormal{search}} \right) \right] \\
&\le \exp\left( k^2 \cdot \psi_{\mP}\left(\lambda\right) - k^2 \lambda \cdot \tau_{\textnormal{search}} \right)
\end{align*}
Again, since $\tau_{\textnormal{search}} \in \left(- \KL(\mQ \| \mP), \KL(\mP \| \mQ) \right)$, we may take $\lambda \le 0$ so that $\lambda \cdot \tau_{\text{search}} - \psi_{\mP}(\lambda)$ is arbitrarily close to $E_\mP(\tau_{\textnormal{search}})$. Therefore
$$\bP_{H_1}\left[ T_{\textnormal{search}}(M) < \tau_{\textnormal{search}} \right] \le \exp\left( - E_\mP(\tau_{\text{search}}) \right) = o(1)$$
which proves the desired result.
\end{proof}

We now show that the search test solves $\pr{ssd}(n, k, \mP, \mQ)$ for $\pr{uc-b}$ in $(\mP, \mQ)$ in the parameter regime described in Section \ref{sec:completephasediagram}.

\begin{corollary}(Statistically Achievability for $\pr{uc-b}$)
Suppose that $(\mP, \mQ)$ is a computable pair in $\pr{uc-b}$. If $\SKL(\mP, \mQ) \ge \frac{c}{k} \log \left( \frac{n}{k} \right)$ for a sufficiently large constant $c > 0$, then $T_{\textnormal{search}}$ with $\tau_{\textnormal{search}} = 0$ solves $\pr{ssd}(n, k, \mP, \mQ)$.
\end{corollary}

\begin{proof}
It suffices to verify the lower bounds on $E_\mP$ and $E_\mQ$ in Proposition \ref{prop:searchtest}. Since $(\mP, \mQ)$ is in $\pr{uc-b}$, by Lemma \ref{lem:UCB} we have that
\begin{align*}
E_\mP(0) &\ge \frac{\left( - \KL(\mP \| \mQ) \right)^2}{4C \cdot \KL(\mP \| \mQ)} = \frac{1}{4C} \cdot \KL(\mP \| \mQ) \ge \frac{3}{k} \log \left( \frac{n}{k} \right) \\
E_\mQ(0) &\ge \frac{\left( \KL(\mQ \| \mP) \right)^2}{4C \cdot \KL(\mQ \| \mP)} = \frac{1}{4C} \cdot \KL(\mQ \| \mP) \ge \frac{3}{k} \log \left( \frac{n}{k} \right)
\end{align*}
Since $\frac{1}{k} \log \left( \frac{n}{k} \right) = \omega(k^{-2})$, applying Proposition \ref{prop:searchtest} now proves the corollary.
\end{proof}

\section{The Universality Classes UC-A, UC-B and UC-C}
\label{sec:universalityclasses}

\subsection{Universality Classes UC-A and UC-B}

The universality class $\pr{uc-b}$ is discussed at length in Section 2.1 and 3 of \cite{hajek2017information}, which introduces it as Assumption 2 in the context of their information-theoretic lower bounds for general submatrix recovery. The provide a means to check whether a pair $(\mP, \mQ)$ belonging to an exponential family is in $\pr{uc-b}$ in Appendix B of \cite{hajek2017information}. The discussion below shows that this method also checks membership in $\pr{uc-a}$ and $\pr{uc-c}$.

Recall that a random variable $X$ is sub-Gaussian if there are $a, b \in \mathbb{R}$ with $b > 0$ such that $\log \bE[e^{\lambda X}] \le a + \bE[X] \cdot \lambda + b\lambda^2$ for all $\lambda \in \mathbb{R}$. Given a computable pair $(\mP, \mQ)$, let $L(x) = \log \frac{d\mP}{d\mQ}(x)$ denote its LLR. The inequalities
\begin{align*}
\psi_{\mP}(\lambda) - \KL( \mP \| \mQ) \cdot \lambda &\le C \cdot \KL( \mP \| \mQ) \cdot \lambda^2 \quad \textnormal{for all } \lambda \in [-1, 0] \\
\psi_{\mQ}(\lambda) + \KL( \mQ \| \mP) \cdot \lambda &\le C \cdot \KL( \mQ \| \mP) \cdot \lambda^2 \quad \textnormal{for all } \lambda \in [-1, 1]
\end{align*}
defining $\pr{uc-b}$ are exactly the inequalities required for the two distributions $L(X)$ where $X \sim \mP$ and $L(X)$ where $X \sim \mQ$ to be sub-Guassian, but only required to hold for $\lambda$ in restricted intervals. Thus $\pr{uc-b}$ is weaker than sub-Gaussianity of $L$ under $\mP$ and $\mQ$. We now observe that it similarly holds that $\pr{uc-a}$ is weaker than sub-Gaussianity of $L$ under $\mP$. If the sub-Gaussianity inequality holds for the following interval
$$\psi_{\mP}(\lambda) - \KL( \mP \| \mQ) \cdot \lambda \le C \cdot \KL( \mP \| \mQ) \cdot \lambda^2 \quad \textnormal{for } \lambda \in [0, \Theta(\log n)]$$
then the same argument showing Property 1 in Lemma \ref{lem:UCB} shows that
$$E_\mP\left( (\lambda + 1) \cdot \KL(\mP \| \mQ) \right) = \Omega\left( \KL(\mP \| \mQ) \cdot (\log n)^2 \right)$$
for $\lambda = \Theta(\log n)$. The convexity of $E_\mP$ and the fact that $E_\mP(\KL(\mP \| \mQ)) = 0$ implies that if $n$ is large enough so that $n^{\epsilon} \ge \lambda + 1$, then
$$\frac{E_\mP\left(n^{\epsilon} \cdot \KL(\mP \| \mQ) \right)}{(n^\epsilon - 1) \cdot \KL(\mP \| \mQ)} \ge \frac{E_\mP\left((\lambda + 1) \cdot \KL(\mP \| \mQ) \right)}{\lambda \cdot \KL(\mP \| \mQ)} = \Omega(\log n)$$
and thus the condition needed for $\pr{uc-a}$ holds. It is also shown in \cite{hajek2017information} that $(\mP, \mQ)$ with bounded LLR are in $\pr{uc-b}$.

\begin{lemma}[Lemma 1 in \cite{hajek2017information}] \label{lem:blllrucb}
If $|L| \le B$ for some constant $B > 0$, then $(\mP, \mQ)$ is in $\pr{uc-b}$ with constant $C = e^{5B}$.
\end{lemma}

We now show that the three pairs of interest $\mD_{\pr{bc}}$, $\mD_{\pr{sp}}$ and $\mD_{\pr{gp}}$ introduced in Section \ref{sec:summary} are in $\pr{uc-a}$ and $\pr{uc-b}$. In Sections 2.1 and 3 of \cite{hajek2017information}, it is shown that all three of these pairs lie in $\pr{uc-b}$. Thus it suffices to verify that they lie in $\pr{uc-a}$. Consider $\mD_{\pr{bc}}$ where $\mP = \mN(\mu, 1)$ and $\mQ = \mN(0, 1)$. As shown in Section 2.1 of \cite{hajek2017information}, we have that
$$E_\mP(\theta) = \frac{1}{8} \left( \mu - \frac{2\theta}{\mu} \right)^2$$
Suppose that $\theta = n^{\epsilon} \cdot \KL(\mP \| \mQ) = \frac{1}{2} n^{\epsilon} \cdot \mu^2$ by a standard formula for the KL divergence between two Gaussians. Then it follows that
$$E_\mP(\theta) = \frac{1}{8} \left( \mu - n^{\epsilon} \cdot \mu \right)^2 = \Theta(n^{2\epsilon} \mu^2)$$
Since $n^{2\epsilon} = \omega(\log n)$, it follows that $\mD_{\pr{bc}}$ is in $\pr{uc-a}$. In the following, we use several computations $\KL$ computations to be carried in Section \ref{sec:uc-c}. As shown in \cite{hajek2017information}, if $\mP = \text{Bern}(p)$ and $\mQ = \text{Bern}(q)$ then
$$E_\mP(\theta) = D(\alpha \| p) \quad \text{where } \alpha = \frac{\theta + \log \frac{1 - q}{1 - p}}{\log \frac{p(1 - q)}{q(1 - p)}}$$
where $D( \cdot \| \cdot)$ is the binary entropy function. Letting $\theta = D(p \| q) + \tau$ yields that
$$E_\mP(D(p \| q) + \tau) = D\left( p + \tau \left( \log \frac{p(1 - q)}{q(1 - p)} \right)^{-1} \Big\| p \right)$$
When $p = cq = cn^{-\alpha}$ for some $c > 1$, it is not difficult to verify that if $\tau$ is such that $\tau = \Theta(n^\epsilon \cdot D(p \| q)) = \Theta(n^{\epsilon - \alpha})$ then
$$E_\mP(D(p \| q) + \tau) = \Theta\left( n^{\epsilon-\alpha} \log n \right)$$
and thus $\mD_{\pr{sp}}$ is in $\pr{uc-a}$. Furthermore, if $p = n^{-\alpha} + \Theta(n^{-\gamma})$ and $q = n^{-\alpha}$ where $\gamma > \alpha > 0$, then $D(p \| q) = \Theta(n^{\alpha - 2\gamma})$. Observe that $\log \frac{p(1 - q)}{q(1 - p)} = \Theta(n^{\alpha - \gamma})$. Now taking $\tau = \Theta(n^\epsilon \cdot D(p \| q)) = \Theta(n^{\epsilon + \alpha - 2\gamma})$ yields that
$$E_\mP(D(p \| q) + \tau) = \Theta \left( \frac{\tau^2}{p} \left( \log \frac{p(1 - q)}{q(1 - p)} \right)^{-2} \right) = \Theta\left( n^{2\epsilon + \alpha - 2\gamma} \right)$$
and since $n^{\epsilon} = \omega(\log n)$, it follows that $\mD_{\pr{gp}}$ is in $\pr{uc-a}$. We conclude this section by generalizing these computations to show that if a pair $(\mP, \mQ)$ has bounded LLR then it is also in $\pr{uc-a}$. The proof is similar to that of Lemma 1 in \cite{hajek2017information}.

\begin{lemma} \label{lem:ucabounded}
If $|L| \le B$ for some constant $B > 0$, then $(\mP, \mQ)$ is in $\pr{uc-a}$.
\end{lemma}

\begin{proof}
First note that if $\lambda \in [0, \lambda_{\max}]$ then
$$\psi''_{\mP}(\lambda) = \frac{\bE_{\mP}[L^2 \cdot \exp(\lambda L)] \cdot \psi_{\mP}(\lambda) - \bE_{\mP}[L \cdot \exp(\lambda L)]^2}{\psi_{\mP}(\lambda)^2} \le \frac{\bE_{\mP}[L^2 \cdot \exp(\lambda L)]}{\bE_{\mP}[\exp(\lambda L)]} \le e^{2B \lambda_{\max}} \cdot \bE_{\mP}[L^2]$$
As in \cite{hajek2017information}, let $\phi(x) = e^x - x - 1$ and note that if $|x| \le B$ then $\frac{1}{2} e^{-B} x^2 \le \phi(x) \le \frac{1}{2} e^B x^2$ since $\phi$ is nonnegative, convex and satisfies $\phi(0) = \phi'(0) = 0$ and $\phi''(x) = e^x \in [e^{-B}, e^B]$ if $|x| \le B$. Therefore we have that
$$\bE_{\mP}[L^2] = \bE_{\mQ} \left[ L^2 \exp(L) \right] \le e^B \cdot \bE_\mQ[L^2] \le 2e^{2B} \cdot \bE_{\mQ}[\phi(L)] = 2e^{2B} \cdot \KL(\mQ \| \mP)$$
Applying Lemma \ref{lem:blllrucb} yields that $(\mP, \mQ)$ is in $\pr{uc-b}$ and thus $\KL(\mQ \| \mP) \le c \cdot \KL(\mP \| \mQ)$ for some $c > 0$ by Property 3 in Lemma \ref{lem:UCB}. Thus $\psi_{\mP}''(\lambda) \le 2c \cdot e^{2B(\lambda_{\max} + 1)} \cdot \KL(\mP \| \mQ)$ for all $\lambda \in [0, \lambda_{\max}]$. Combining this with $\psi_{\mP}'(0) = \KL(\mP \| \mQ)$ and $\psi_{\mP}(0) = 0$ yields that
$$\psi_{\mP}(\lambda) \le \KL(\mP \| \mQ) \cdot \lambda + c \cdot e^{2B(\lambda_{\max} + 1)} \cdot \KL(\mP \| \mQ) \cdot \lambda^2 \quad \text{for all } \lambda \in [0, \lambda_{\max}]$$
This inequality implies that
$$E_{\mP}(\KL(\mP \| \mQ) + \tau) \ge \sup_{\lambda \in [0, \lambda_{\max}]} \left\{ \tau \cdot \lambda - c \cdot e^{2B(\lambda_{\max} + 1)} \cdot \KL(\mP \| \mQ) \cdot \lambda^2 \right\} \ge \frac{\tau^2}{4c \cdot e^{2B(\lambda_{\max} + 1)} \KL(\mP \| \mQ)}$$
where the last inequality holds as long as
$$\frac{\tau}{2c \cdot e^{2B(\lambda_{\max} + 1)} \KL(\mP \| \mQ)} \le \lambda_{\max}$$
Now take $\tau = (n^{\epsilon} - 1) \cdot \KL(\mP \| \mQ)$ and $\lambda_{\max} + 1 = \frac{1}{2B}[\epsilon \log n - \log \log n] + c'$ for some constant $c' > 0$. It follows that
$$\frac{\tau}{2c \cdot e^{2B(\lambda_{\max} + 1)} \KL(\mP \| \mQ)} \le \frac{e^{-2B c'}}{2c} \cdot \log n$$
which is at most $\lambda_{\max}$ for a large enough choice of $c' = c'(B, c) > 0$. Furthermore, substituting this pair $(\tau, \lambda_{\max})$ into the inequality above yields that $E_{\mP}(n^{\epsilon} \cdot \KL(\mP \| \mQ)) = \Omega(n^{\epsilon} \cdot \KL(\mP \| \mQ) \cdot \log n)$, which implies that $(\mP, \mQ)$ is in $\pr{uc-a}$.
\end{proof}

\subsection{Universality Class UC-C}
\label{sec:uc-c}

The condition for $(\mP, \mQ)$ to be in $\pr{uc-c}$ is also weaker than sub-Gaussianity of the LLR. Observe that if the sub-Gaussian inequality
$$\psi_{\mQ}(\lambda) + \KL( \mQ \| \mP) \cdot \lambda \le C \cdot \KL( \mQ \| \mP) \cdot \lambda^2$$
holds for $\lambda = 2$, then $\chi^2(\mP \| \mQ) = \psi_\mQ(2) \le (4C - 2) \cdot \KL(\mQ \| \mP)$ and $(\mP, \mQ)$ is in $\pr{uc-c}$. In light of this, taking $\lambda_{\max} \ge 2$ in the proof of Lemma \ref{lem:ucabounded} also shows that $|L| \le B$ implies that $(\mP, \mQ)$ is in $\pr{uc-c}$. For the sake of the exposition, we now verify directly that the three pairs $\mD_{\pr{bc}}$, $\mD_{\pr{sp}}$ and $\mD_{\pr{gp}}$ are in $\pr{uc-c}$ by explicitly computing their KL and $\chi^2$ divergences.
\begin{enumerate}
\item[($\mD_{\textsc{bc}}$)] If $\mP = \mN(\mu, 1)$ and $\mQ = \mN(0, 1)$ where $\mu = n^{-\alpha}$ for some $\alpha > 0$, then we have that
\begin{align*}
\chi^2(\mP \| \mQ) &= \frac{1}{2} \left( e^{\mu^2} - 1 \right) = \Theta(\mu^2) \\
\SKL(\mP, \mQ) &= \KL(\mQ \| \mP) + \KL(\mP \| \mQ) = \mu^2
\end{align*}
as $\mu \to 0$, by well-known formulas for these divergences.
\item[($\mD_{\textsc{sp}}$)] If $\mP = \text{Bern}(p)$ and $\mQ = \text{Bern}(q)$ where $p = cq = cn^{-\alpha}$ for some constant $c > 1$ and $\alpha > 0$, then we have that
\begin{align*}
\chi^2( \mP \| \mQ ) &= \frac{(p - q)^2}{q(1 - q)} = \frac{(c - 1)^2}{1 - n^{-\alpha}} \cdot n^{-\alpha} = \Theta(n^{-\alpha}) \\
\KL(\mQ \| \mP) &= - n^{-\alpha} \log c - (1 - n^{-\alpha}) \log \left( 1 - \frac{(c - 1)n^{-\alpha}}{1 - n^{-\alpha}} \right) \\
&= -n^{-\alpha} \log c - (1 - n^{-\alpha}) \cdot \left( -\frac{(c - 1)n^{-\alpha}}{1 - n^{-\alpha}} + O(n^{-2\alpha})\right) \\
&= (c - 1 - \log c) n^{-\alpha} + O(n^{-2\alpha}) = \Theta(n^{-\alpha}) \\
\KL(\mP \| \mQ) &= (c^{-1} - 1 - \log c^{-1}) cn^{-\alpha} + O(n^{-2\alpha}) = \Theta(n^{-\alpha})
\end{align*}
\item[($\mD_{\textsc{gp}}$)] If $\mP = \text{Bern}(p)$ and $\mQ = \text{Bern}(q)$ where $p = q + \Theta(n^{-\gamma})$ and $q = n^{-\alpha}$ for some constants $\gamma > \alpha > 0$, then we have that
\begin{align*}
\chi^2( \mP \| \mQ ) &= \frac{(p - q)^2}{q(1 - q)} = \Theta(n^{-2\gamma +\alpha}) \\
\KL(\mQ \| \mP) &= - q \log \left( 1 + \frac{p - q}{q} \right) - (1 - q) \log \left( 1 - \frac{p - q}{1 - q} \right) \\
&= -q \left( \frac{p - q}{q} - \left( \frac{p - q}{q} \right)^2 + O\left( \frac{(p - q)^3}{q^3} \right) \right) \\
&\quad \quad - (1 - q) \left( - \frac{p - q}{1 - q} - \left( \frac{p - q}{1 - q} \right)^2 + O\left( \frac{(p - q)^3}{(1 - q)^3} \right) \right) \\
&= \frac{(p - q)^2}{q(1 - q)} + O(n^{-3\gamma+2\alpha}) = \Theta(n^{-2\gamma + \alpha}) \\
\KL(\mP \| \mQ) &= \Theta(n^{-2\gamma + \alpha})
\end{align*}
\end{enumerate}
These computations verify that all three pairs $(\mP, \mQ)$ are in the universality class $\pr{uc-c}$.

\section{Further Questions}
\label{sec:conc}

This work leaves a number of questions about submatrix detection and planted clique reductions unresolved. The following is an overview of some of these problems.
\begin{itemize}
\item \textbf{Weaker Universality Assumptions:} Can our required lower bound on $E_\mP$ in our main computational lower bounds be relaxed? In other words, is there a reduction from planted clique or another conjecturally hard average-case problem to the general submatrix detection problem for a wider universality class of $(\mP, \mQ)$?
\item \textbf{Other Possible Computational Phase Diagrams:} Outside of our universality classes, are there any natural universality classes $(\mP, \mQ)$ with different phase diagrams that can be characterized through average-case reductions? One example of a pair $(\mP, \mQ)$ outside of our universality classes that we do not show hardness for is $\mP = \text{Bern}(p)$ and $\mQ = \text{Bern}(q)$ where $p = n^{-\alpha}$ and $q = n^{-\beta}$ where $\beta > \alpha$. The graph variant of this submatrix detection problem corresponds to the log-density regime of planted dense subgraph and seems to obey a completely different phase diagram. Algorithms and conjectured hardness for this problem are discussed in \cite{bhaskara2010detecting,chlamtac2012everywhere, chlamtavc2017minimizing,chlamtac2018sherali}.
\item \textbf{Computational Lower Bounds for Submatrix Recovery:} Through similar detection-recovery reductions as in Section 10 of \cite{brennan2018reducibility}, our computational lower bounds for submatrix detection yields computational lower bounds for the general recovery variant. However, $T_{\text{sum}}$ does not translate into a natural recovery algorithm and the computational barrier for recovery appears to be different from that of submatrix detection. This has left a region of the phase diagram with an unknown computational complexity. Semidefinite programming algorithms for recovery under regularity assumptions on $(\mP, \mQ)$ were analyzed in \cite{hajek2016semidefinite} meeting the polynomial time threshold shown in Figure \ref{fig:recoverydiagram}. In a distributionally robust sub-Gaussian variant of submatrix recovery, planted clique lower bounds were shown by \cite{cai2017computational}. The known and open regions of the phase diagram for recovery are shown in Figure \ref{fig:recoverydiagram}.
\end{itemize}

\begin{figure*}[t!]
\centering
\begin{tikzpicture}[scale=0.45]
\tikzstyle{every node}=[font=\footnotesize]
\def\xmin{0}
\def\xmax{16}
\def\ymin{-1}
\def\ymax{11}

\draw[->] (\xmin,\ymin) -- (\xmax,\ymin) node[right] {$\beta$};
\draw[->] (\xmin,\ymin) -- (\xmin,\ymax) node[above] {$\alpha$};

\node at (15, -1) [below] {$1$};
\node at (7.5, -1) [below] {$\frac{1}{2}$};
\node at (0, 0) [left] {$0$};
\node at (0, 10) [left] {$2$};
\node at (0, 5) [left] {$1$};
\node at (0, 3.33) [left] {$\frac{2}{3}$};
\node at (10, -1) [below] {$\frac{2}{3}$};

\filldraw[fill=red!25, draw=red] (7.5, 0) -- (15, 5) -- (10, 3.33) -- (7.5, 0);
\filldraw[fill=cyan!25, draw=blue] (0, 0) -- (7.5, 0) -- (10, 3.33) -- (0, 0);
\filldraw[fill=green!25, draw=green] (0, 0) -- (7.5, 0) -- (15, 5) -- (15, -1) -- (0, -1) -- (0, 0);
\filldraw[fill=gray!25, draw=gray] (0, 0) -- (15, 5) -- (15, 10) -- (0, 10) -- (0, 0);

\node at (4, -0.5) {$\SKL(\mP, \mQ) \asymp 1$};
\node at (12.3, 1.1) {$\SKL(\mP, \mQ) \asymp \frac{n}{k^2}$};
\node at (4, 3) {$\SKL(\mP, \mQ) \asymp \frac{1}{k}$};
\node at (7.5, 9) {information-theoretically impossible};
\node at (12, -0.25) {poly-time};
\node at (6.5, 1.25) {PC-hard};
\node at (11, 3) {open};
\end{tikzpicture}

\caption{Computational and statistical barriers in the recovery variant of $\pr{ssd}(n, k, \mP, \mQ)$ under regularity assumptions on $(\mP, \mQ)$. The red region is conjectured to be computationally hard but no $\textsc{pc}$ reductions showing this hardness are known. Axes are parameterized as $k = \tilde{\Theta}(n^{\beta})$ and $\SKL(\mP, \mQ) = \tilde{\Theta}(n^{-\alpha})$.}
\label{fig:recoverydiagram}
\end{figure*}
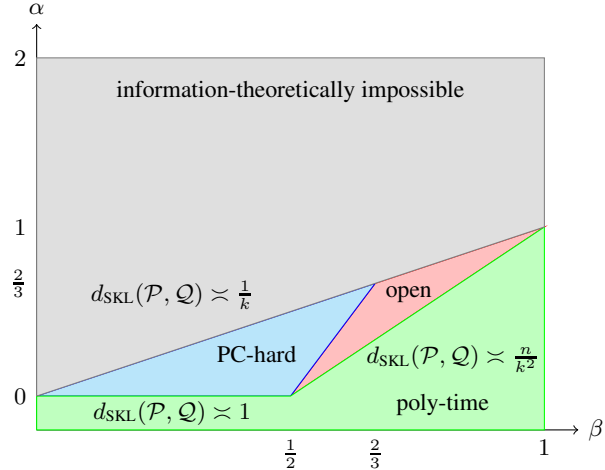

\section*{Acknowledgements}

We thank Philippe Rigollet, Yury Polyanskiy, Ankur Moitra, Elchanan Mossel, Jonathan Weed, Frederic Koehler, Enric Boix Adser\`{a}, Austin J. Stromme, Vishesh Jain and Yash Deshpande for inspiring discussions on related topics. This work was supported in part by the grant ONR N00014-17-1-2147.

\bibliography{GB_BIB.bib}
\bibliographystyle{alpha}

\end{document}